\documentclass[a4paper]{article}

\usepackage[english]{babel}

\usepackage[utf8]{inputenc}
\usepackage[utf8]{inputenc}
\usepackage{amsmath}
\usepackage{graphicx}
\usepackage{amssymb}
\usepackage{amsthm}
\usepackage{tikz-cd}
\usepackage{mathrsfs}
\usepackage[colorinlistoftodos]{todonotes}
\usepackage{enumitem}
\usepackage{yfonts}
\usepackage{ dsfont }
\usepackage{MnSymbol}
\usepackage{slashed}

\title{ Mukai duality on K3 surfaces from the differential geometric perspective }

\author{Yang Li 	\thanks{Y.L. is supported by the Engineering and Physical Sciences Research Council [EP/L015234/1], the EPSRC Centre for Doctoral Training in Geometry and Number Theory (The London School of Geometry and Number Theory), University College London. The author is also funded by Imperial College London for his PhD studies. }}

\date{\today}
\newtheorem{thm}{Theorem}[section]
\newtheorem{lem}[thm]{Lemma}

\theoremstyle{definition}

\newtheorem{cor}[thm]{Corollary}

\newtheorem{rmk}{Remark}
\newtheorem{prop}[thm]{Proposition}
\newtheorem{Def}[thm]{Definition}

\newtheorem*{Acknowledgement}{Acknowledgment}

\newcommand{\ie}{\emph{i.e.} }
\newcommand{\cf}{\emph{cf.} }

\newcommand{\R}{\mathbb{R}}
\newcommand{\C}{\mathbb{C}}
\newcommand{\Z}{\mathbb{Z}}

\newcommand{\Q}{\mathbb{Q}}

\newcommand{\Lap}{\Delta}

\DeclareMathOperator{\Hom}{Hom}
\DeclareMathOperator{\End}{End}
\DeclareMathOperator{\Lie}{Lie}

\DeclareMathOperator{\Tr}{Tr}

\begin{document}
	\maketitle
	
\begin{abstract}
This paper treats the theory of Mukai duality on K3 surfaces from the differential geometric perspective,  taylored to the need of the author's companion paper about Mukai duality of adiabatic coassociative K3 fibrations.
\end{abstract}

\section{Introduction}

This paper gives a differential geometric account of  Mukai duality on K3 surfaces, intended as a background reference paper for  the author's companion paper \cite{myG2paper} dealing with Mukai duality of adiabatic coassociative K3 fibrations.

Mukai duality in the algebro-geometric setting has a well established literature \cite{Huy1}\cite{FourierMukai}\cite{Mukai}\cite{Mukai1}\cite{Mukai2}\cite{Mukai3}. The basic picture is that moduli spaces of stable vector bundles of certain topological types (specified by a Mukai vector $v(E)$) on a K3 surface $X$, are again K3 surfaces $X^\vee$, whose periods and symplectic forms can be described in terms of data on $X$. Morever, one can use the (quasi)-universal bundle to define the Fourier-Mukai transform (a.k.a.  Nahm transform), which under appropriate cohomology vanishing conditions  converts vector bundles on $X$ to vector bundles on $X^\vee$ and vice versa.

On the differential geometric side, the papers of Bartocci et al \cite{Bartocci1}\cite{Bartocci2} deal with the same subject and share some of the key intermediate results with this paper, although their technical emphasis is on twistor methods and reduction to complex geometry, while this paper primarily uses spinor techniques and can be more readily adapted to the $G_2$ setting \cite{myG2paper}. 
Our setting is also more general as we do not assume the Mukai vector to have zero degree.
 A major influence to this paper is the work of Braam and Baal \cite{BraamBaal} on the Nahm transform over a 4-torus.

The outline of this paper is as follows (for more details, see introductions to each Chapter):

 Chapter \ref{ModuliofvectorbundlesonK3} begins with a brief review of the algebro-geometric theory, and recalls the celebrated construction of the hyperk\"ahler structure on $X^\vee$. We then show carefully how to put an optimal connection $\nabla^{univ}$ on the universal bundle $\mathcal{E}\to X\times X^\vee$.

Chapter \ref{TheMukaidualofaK3surface} studies several aspects of the geometry of the Mukai dual K3 surface $X^\vee$: the hyperk\"ahler periods of $X^\vee$, the variation of Hermitian-Yang-Mills connections over $X^\vee$ parametrised by $X$, and the interpretation of the hyperk\"ahler structure on $X$ in terms of data on $X^\vee$.

Chapter \ref{TheNahmtransformonK3surfaces} studies the Nahm transform using differential geometric techniques, and is the more original part of this paper. The main result says

\begin{thm}(\cf Theorem \ref{NahmtransformperservesASD} and \ref{Fourierinversiontheorem})
Let $(\mathcal{F}, \alpha)$ be an irreducible HYM connection over a hyperK\"ahler K3 surface $X$, whose Mukai vector $v(\mathcal{F})$ shares the same slope as  $v(E)$ but $v(\mathcal{F})\neq v(E)$. Then the Nahm transform is well defined and produces a HYM connection $(\hat{\mathcal{F}}, \hat{\alpha})$ over $X^\vee$. Under further nonsingularity assumptions the inverse Nahm transform $(\hat{\hat{\mathcal{F}}}, \hat{\hat{\alpha}})$ of $(\hat{\mathcal{F}}, \hat{\alpha})$ is well defined and is canonically isomorphic to $(\mathcal{F},\alpha)$.

\end{thm}

\begin{rmk}
Here the slope condition ensures the Hom bundle carries ASD connections rather than just HYM, and $v(\mathcal{F})\neq v(E)$ means the Fourier-Mukai transform is a bundle, not a skyscraper sheaf.	
\end{rmk}

\begin{rmk}
Most results in the present paper will be applied to \cite{myG2paper}. For readers interested in \cite{myG2paper}, familiarity with Chapter \ref{ModuliofvectorbundlesonK3} and Section \ref{HyperkahlerperiodsontheMukaidualK3}, \ref{ASDconnectionsonMukaidual}, \ref{Nahmtransform} is essential, and knowledge of other parts is useful. 	
\end{rmk}

\begin{Acknowledgement}
The author thanks his PhD supervisor Simon Donaldosn and co-supervisor Mark Haskins for inspirations, and Simon Center for hospitality.

\end{Acknowledgement}

\section{Moduli of vector bundles on K3 surfaces}\label{ModuliofvectorbundlesonK3}

Moduli of vector bundles on K3 surfaces, and Fourier-Mukai transforms in particular, has been studied extensively in algebraic geometry in the language of coherent sheaves and derived categories. From the viewpoint of our applications, it seems preferable to interpret this theory in differential geometry in terms of bundles and ASD instantons; this has the advantage of not favoring any particular complex structure, and makes easier contact with metric geometry and $G_2$ instantons. The fundamental link between these two viewpoints is the Hitchin-Kobayashi correspondence.

Section \ref{Mukaidualreview} reviews some basic results about moduli spaces of stable vector bundles on K3 surfaces, which in some special cases turn out to be again K3 surfaces, called the Mukai dual. Section  \ref{hyperkahlerquotientreview} reviews the celebrated construction of the hyperk\"ahler metric on these moduli spaces, from a gauge theoretic viewpoint. These materials are standard, which we include in the hope of making the paper more accessible. Section \ref{universalconnection} continues with the gauge theoretic thread, and explain under some topological conditions, how to construct a universal connection on the universal bundle, which is compatible with all the 3 complex structures, and restricts to ASD connections on each K3 fibre. When part of the topological conditions fail, we show how to modify the construction to achieve a second best substitute.

\begin{rmk}
The construction of the universal connection is also an important step in the work of Bartocci et al \cite{Bartocci1}, but these authors applied a result which only works for semisimple structure groups, and missed out the topological issues arising from the centre of the $U(r)$ structure group, as we will discuss in Section \ref{universalconnection}.
\end{rmk}

\subsection{The Mukai dual of a K3 surface}\label{Mukaidualreview}

This Section gives a very brief review of the standard algebro-geometric theory, mainly following \cite{Huy1}, Section 6.1, which in turn is based on the works of Mukai.
Let $(X, \omega_1, \omega_2, \omega_3)$ be a \textbf{hyperk\"ahler K3 surface}. In this Section we consider the preferred complex structure $\omega_2+\sqrt{-1} \omega_3$. 

\begin{rmk}
The algebraic theory requires the projectivity assumption, \ie the K\"ahler class of $\omega_1$ is rational. This is not restrictive because the 2-sphere of complex structures always contain projective members.

\end{rmk}

It is convenient in the study of  bundles on $X$ to incorporate the topological data into the so called Mukai vector:

\begin{Def}
The \textbf{Mukai vector} of a coherent sheaf $E$ on a smooth variety $V$ is $v(E)=ch(E) \sqrt{Td(V)} \in H^{2 *}(V, \Q)$. More concretely, for $V=X$ the K3 surface, let $r=rk(E), c_1=c_1(E), c_2=c_2(E)$, then $v(E)=(r, c_1, \frac{1}{2}c_1^2-c_2+r)$.
\end{Def}
\begin{Def}
	If $v=\oplus v_i \in \bigoplus H^{2i}(V, \Z)$, then let $v^\vee=\oplus (-1)^i v_i$. We define the \textbf{Mukai pairing} to be the bilinear form on $H^{2*}(V,\Q)$ 
	\[
	(v,w)= -\int_V v^\vee \cup w.
	\]
\end{Def}

The importance of the Mukai vector comes from the Riemann-Roch formula:
\begin{lem}(\cf \cite{Huy1} Page 168)
If $E$, $F$ are coherent sheaves on $X$ then the Euler characteristic of the pair $(E,F)$ is 
\[
\chi(E, F)= \sum (-1)^i \dim \text{Ext}^i(E,F)=- (v(E), v(F) ).
\]
\end{lem}
We consider the (coarse) moduli space $\mathcal{M}(v)=M(r,c_1,c_2)$ of semistable sheaves with a given Mukai vector $v$, and we denote the locus of stable vector bundles inside by $\mathcal{M}^s(v)$.
Now if $E$ is a stable vector bundle, then by Serre duality $\text{H}^0(\End_0(E))=\text{H}^2(\End_0(E))=0$, so the moduli space is smooth at the point $E$ with dimension $(v,v)+2$.

A feature of the theory is that low dimensional moduli spaces are easier to understand.

\begin{thm}\label{Mukaidualalgebraic}
(\cf \cite{Huy1} page 169) Assume $(v,v)=0$. If $\mathcal{M}^s(v)$ has a compact irreducible component, then  $\mathcal{M}^s(v)$ is equal to this component.
\end{thm}

\begin{Def}\label{Mukaidualdefinition}
Assume $(v,v)=0$, the moduli space $\mathcal{M}^s(v)$ is compact and nonempty. Then we call $X^\vee=\mathcal{M}^s(v)$ the \textbf{Mukai dual} of $X$.
\end{Def}

An important tool is Fourier-Mukai transform (FM) on cohomology. This is easiest to explain when the universal bundle $\mathcal{E}\to X\times X^\vee$ exists, but even if it does not, FM can still be defined via the `quasi-universal family' which always exists (\cf Section 4.6 \cite{Huy1}).
$FM$ can be thought of as the analogue of the Fourier transform and the inverse Fourier transform.

\begin{Def}\label{DefinitionofFourierMukai}
The \textbf{Fourier-Mukai transform on cohomology} is a pair of maps between the cohomologies of $X$ and $X^\vee$, denoted  $FM: H^*(X, \Q)\to H^*(X^\vee, \Q)$ and $FM^\vee: H^*(X^\vee, \Q)\to H^*(X, \Q)$. Let $pr_X: X\times X^\vee\to X$, $pr_{X^\vee}: X\times X^\vee\to X^\vee$ be the two projections, and $v(\mathcal{E})$ be the Mukai vector of $\mathcal{E}$ on $X\times X^\vee$. Then 
\[
FM(\alpha)=pr_{X^\vee_*}( v(\mathcal{E})^\vee \cup pr_{X}^*\alpha  ), \quad FM^\vee(\alpha')=pr_{X_*}( v(\mathcal{E}) \cup pr_{X^\vee}^*\alpha'  ).
\]
\end{Def}
\begin{rmk}\label{FourierMukaicohomology}
The formula is motivated by compatibility with the Fourier-Mukai transform on bundles/ sheaves/ derived category/ K-theory, which also involves this kind of convolution operations.  The basic idea is to start from any bundle/ coherent sheaf/complex of sheaves $\mathcal{F}\to X$, and the Fourier-Mukai transform will output the complex $R^*pr{X^\vee_*}( \mathcal{E}^\vee \otimes pr_{X}^*\mathcal{F}  )   $. The alternating sum 
\[
[R^0pr_{X^\vee_*}( \mathcal{E}^\vee \otimes pr_{X}^*\mathcal{F}  ) ]-[R^1pr_{X^\vee_*}( \mathcal{E}^\vee \otimes pr_{X}^*\mathcal{F}  ) ]+[R^2pr_{X^\vee_*}( \mathcal{E}^\vee \otimes pr_{X}^*\mathcal{F}  ) ]
\]
defines a K-theory class, and therefore its Chern character defines a cohomology class. The Mukai vector on $X^\vee$ constructed from this Chern character is related to the Mukai vector of $\mathcal{F}$ by the Fourier-Mukai transform on cohomology (\cf Corollary 5.29 in \cite{FourierMukai}). In many interesting situations, $R^0$ and $R^2$ vanish, and then the Fourier-Mukai transform on cohomology can be used to work out the Chern character of the $R^1$ term, which in good situations is a bundle over $X^\vee$.

\end{rmk}

\begin{Def}
The natural weight two Hodge structure on the even degree cohomology $H^{even}(Y)$ of a compact complex surface $Y$, is given by prescribing $H^{even 2,0}(Y)= H^{2,0}(Y)$, $H^{even 0,2}(Y)= H^{0,2}(Y) $, and $H^{even 1,1}(Y)=H^0(Y)\bigoplus H^{1,1}(Y) \bigoplus H^4(Y)$.
\end{Def}

An important characterisation of the complex structure of $X^\vee$ is
\begin{thm}\label{FourierMukaitransformoncohomologymainperoperties}
(\cf \cite{Huy1}, Section 6.1) The Mukai dual $X^\vee$ is also a K3 surface. The Fourier-Mukai transform pair  $FM$ and $FM^\vee$ are inverse to each other, define an isomorphism of weight two Hodge structures $H^*(X) \simeq H^*(X^\vee)$, preserve the Mukai pairing, and are muturally adjoint with respect to the Mukai pairing. 
\end{thm}

\begin{thm}(\cite{Huy1}, Proposition 6.1.14) 
The Fourier-Mukai transform $FM$ sends the orthogonal complement of $v$ in $H^*(X)$ to $H^2(X^\vee)\bigoplus H^4(X^\vee)$, and sends the isotropic vector $v$ to the fundamental class $[X^\vee]^*\in H^4(X^\vee)$, so induces a linear isometry $H^2(X^\vee, \R)\simeq v^\perp/\R v.$
If $v$ is furthermore primitive, \ie not divisible by any nontrivial integer, then $FM$ induces an isomorphism of Hodge structures $H^2(X^\vee, \Z)\simeq v^\perp /v\Z$.	
\end{thm}

\begin{rmk}
A crucial fact in the proof is that $X^\vee$ admits a hyperk\"ahler structure (see next Section), which by the classification of surfaces leaves only the possibility of K3 surfaces and Abelian surfaces. One uses further information from the cohomology of $X^\vee$ obtained by studying the transform $FM$ to show it is K3. The theorem characterises the period data of the complex structure of $X^\vee$, so by the Torelli theorem determines the complex structure.
\end{rmk}

\subsection{Hyperk\"ahler structure on moduli of vector bundles}\label{hyperkahlerquotientreview}

Fix a Hermitian vector bundle $E$ with Mukai vector $v$. The Hitchin-Kobayashi correspondence allows for comparison between the moduli space $\mathcal{M}^s(v)$ of stable holomorphic vector bundles structures on $E$, and Hermitian-Yang-Mills (HYM) connections $A$. We shall take a viewpoint where \emph{no complex structure is preferred}.

Since K3 surfaces have no torsion cohomology, any $PU(r)$-bundle lifts topologically to a $U(r)$-bundle. HYM connections are equivalent to saying the associated $PU(r)$-connections are ASD, and the central curvature satisfies
$
\frac{\sqrt{-1}}{2\pi r}\Tr F_A= \mathcal{B},
$
where $\mathcal{B}$ is the harmonic representative of $\frac{1}{r}c_1(E)$.

 Hitchin's hyperk\"ahler quotient leads to the following celebrated result:

\begin{thm}(see \cite{Mukai})\label{hyperkahlerstructureonmodulitheorem}
On $\mathcal{M}^s(v)$, there is a canonical hyperk\"ahler structure.
\end{thm}

\begin{cor}
The Mukai dual $X^\vee$ has a natural \textbf{hyperk\"ahler structure}.
\end{cor}

\begin{proof}
(Theorem  \ref{hyperkahlerstructureonmodulitheorem})

Consider the affine space $\mathcal{A}$ of projective unitary connections on $E$, which has a Euclidean hyperk\"ahler structure
$(\mathcal{A}, \overline{g}, \overline{\omega_1}, \overline{\omega_2}, \overline{\omega_3})$. The metric is given by
\[
\overline{g}(a, b)=\frac{1}{4\pi^2} \int_X  \langle a, b\rangle d\text{Vol}_X, \quad a\in T_A \mathcal{A},
\]
where $T_A \mathcal{A}$ is identified as traceless $ad(E)$ valued 1-forms, and the pointwise inner product is defined using the negative of the $su(r)$ trace pairing and the inner product on 1-forms. The hyperk\"ahler forms are
\[
\overline{\omega}_i (a, b)= \frac{-1}{4\pi^2}\int_X \Tr (a \wedge b) \wedge \omega_i, \quad a, b\in T_A\mathcal{A}.
\]
The corresponding complex structures $I_i a$ are simply acting pointwise on the 1-form part of $a$ by the negative of precomposition:
\[
I_i a= -a\circ I_i, \quad a\in T_A\mathcal{A}.
\]
Therefore $\overline{\omega_i}(\cdot{}, \cdot{})=\overline{g}(I_i \cdot{}, \cdot{})$.

The space $\mathcal{A}$ admits an action by the gauge group $\mathcal{G}$ of $PU(r)$ gauge transformations. The hyperk\"ahler moment map is  $\mu=(\mu_1, \mu_2, \mu_3)$, where
\[
\mu_i=\frac{-1}{4\pi^2} F_A \wedge \omega_i.
\] 
The zeros of $\mu$ are just the ASD connections on $X$. At a point $A$ representing an irreducible ASD connection, the group action is free, and the moment map $\mu$ is regular at $A$. We have $X^\vee= \mu^{-1}(0) /\mathcal{G}$, which equips $X^{\vee}$ with a hyperk\"ahler structure.

For applications in this paper, it is useful to recall the description of the hyperk\"ahler structure on the quotient. At a smooth point $A$, there is a canonical orthogonal decomposition of vector spaces
\begin{equation}\label{hyperkahlerquotientconstructiondecompositionformula}
\begin{split}
T_A \mathcal{A}=\Omega^1(X, ad_0 (E)  )= & T_A\mathcal{M}  \bigoplus \\ & (\Lie \mathcal{G})  A \oplus I_1(\Lie \mathcal{G})  A \oplus I_2(\Lie \mathcal{G})  A \oplus I_3(\Lie \mathcal{G})  A.
\end{split}
\end{equation}
Here $ad_0(E)$ means the traceless part of $ad(E)$, and the tangent space for $\mathcal{M}$ is identified with the finite dimensional vector space of solutions to the linearised ASD equation and the Coulumb gauge condition
\begin{equation}
T_A \mathcal{M}= \{ a\in T_A\mathcal{A} : \quad d^+_A a=0, d_A^* a=0                   \}.
\end{equation}
This is a module of the quaternionic action. The Lie algebra $\Lie \mathcal{G}$ acts at $A$ and the deformation it generates are of the form $d_A \Phi$ for some $\Phi \in \Omega^0(X, ad_0(E))$; these are the elements of $(\Lie \mathcal{G}) A$. The hyperk\"ahler structure $(g^\mathcal{M}, \omega_1^\mathcal{M}, \omega_2^\mathcal{M}, \omega_3^\mathcal{M}) $ on the tangent space $T_A\mathcal{M}$ is then the natural restriction of the Euclidean hyperk\"ahler structure.

The  Levi-Civita connection on the moduli space $\mathcal{M}$ is described as follows. Let $a'$ be a tangent vector field defined on a local open set $T\subset \mathcal{M}$ with coordinates $\tau_i$. We represent this $a'$ as a map from $T$ to the infinite dimensional vector space $\Omega^1(X, ad_0(E))$, which at any point $\tau\in T$ lands in the corresponding tangent space $T_A\mathcal{M}\subset T_A\mathcal{A}=\Omega^1(X, ad_0 (E)  )$. Then to compute the Levi-Civita connection $\nabla^{L.C.}_{\frac{\partial}{\partial \tau_i}  }a$, we first calculate the derivative $\frac{\partial a'}{\partial \tau_i}$, and then orthogonally project to $T_A\mathcal{M}$. This turns out to be well defined.
\end{proof}

\subsection{The universal connection}\label{universalconnection}

The concept of universal bundle in differential geometry and algebraic geometry has a subtle difference. We fix a Mukai vector $v=v(E)$ with $(v,v)=0$, which determines the topological type of a Hermitian bundle $E\to X$. The moduli space of irreducible HYM connections (assumed to be compact and nonempty) must be a K3 surface, called the Mukai dual K3 surface $X^\vee$. The associated universal bundle of $PU(r)$ ASD connections exist unconditionally over $X\times X^\vee$. Since $X\times X^\vee$ has no torsion cohomology, the $PU(r)$ bundle lifts topologically to a $U(r)$ vector bundle $\mathcal{E}\to X\times X^\vee$, and the connections on $X$ fibres can be lifted to the tautological HYM connections on $E\to X$. This $\mathcal{E}\to X\times X^\vee$ is the \textbf{universal bundle} in differential geometry which exists unconditionally, wheras in algebraic geometry compatibility with complex structures on $X^\vee$ imposes further conditions on $c_1(\mathcal{E})$ which are not always satisfied.

Our aim is to put an optimal global connection on $\mathcal{E}$, which restricts fibrewise to the  HYM connection on $X$, by adapting \cite{DonaldsonKronheimer} Section 5.2.3. The na\"ive idea is to start from the space of all unitary connections $\mathcal{A}$, and seek a canonical connection on the principal $U(r)$ bundle $P\times\mathcal{A}$ over $X\times\mathcal{A}$, where $P\to X$ is the principal bundle corresponding to $E\to X$. The canonical connection should be invariant under the action of the gauge group of unitary transformations $\mathcal{G}$, and one tries to descend it to the quotient bundle over $\mathcal{A}/\mathcal{G}$. This is the approach in \cite{Bartocci1}, but it contains missing steps because the quotient bundle is only a $PU(r)$ bundle due to the nontrivial centre in $U(r)$.

To remedy this, we first replace $\mathcal{A}$ and $\mathcal{G}$ by their projective unitary counterparts. We will tacitly restrict attention to the locus of irreducible connections. Now we define the universal connection $\nabla^{univ}$ on $P\times \mathcal{A}\to X\times\mathcal{A}$. At the point $(x, A)\in X\times \mathcal{A}$, we put
\begin{equation}\label{universalconnectionansatz}
\begin{cases}
\nabla^{univ}_v= \nabla^A_v &\quad v\in T_xX, \\
\nabla^{univ}_a= \nabla^{trivial}_a+ G_A d_A^* a &\quad a\in T_A\mathcal{A},
\end{cases}
\end{equation}
where $G_A$ is the inverse of the Laplacian $\Lap_A=d_A^*d_A$, which is well defined because of the irreducible connection condition. If we think of the universal connection as a $u(r)$-valued 1-form on $P\times \mathcal{A}$, then the alternative definition is 
\[
\begin{cases}
A^{univ}|_{P\times\{A\}}= A, \\
A^{univ}(a)= G_A d_A^* a, \quad a\in T_A\mathcal{A}.
\end{cases}
\]
This $\nabla^{univ}$ is $\mathcal{G}$ invariant, so descends to the quotient bundle over $\mathcal{A}/\mathcal{G}$.

\begin{lem}(\cf \cite{DonaldsonKronheimer} Proposition 5.2.17 with minor modifications)
The curvature $F(\nabla^{univ})$ at the point $(x,A)\in X\times \mathcal{A}$ is given by
\begin{equation}\label{universalbundlecurvature}
\begin{cases}
F( \nabla^{univ} ) (u_1, u_2)= F_A(u_1, u_2), &\quad u_1, u_2\in T_x X, \\
F(  \nabla^{univ}  )( a, u   )=\langle a, u\rangle, &\quad a\in T_A \mathcal{A}, u\in T_x X, \quad  d_A^*a=0, \\
F(    \nabla^{univ})(a_1, a_2)= -2 G_A \{ a_1, a_2 \}, &\quad a_1, a_2 \in T_A \mathcal{A},\quad d_A^*a_1=d_A^*a_2=0.
\end{cases}
\end{equation}
where $\{ a_1, a_2\}$ means pointwise taking the Lie bracket on the bundle part, and contracting the 1-form part using the metric on $X$.
\end{lem}

Now given the universal family $\mathcal{E}\to X\times X^\vee$, we get a family of irreducible $PU(r)$ ASD connections, hence a map into this quotient bundle, so we can pullback $\nabla^{univ}$ to obtain a connection on the associated $PU(r)$ bundle of $\mathcal{E}$, also denoted $\nabla^{univ}$, with the same curvature formula.

\begin{prop}
The curvature of the $PU(r)$ connection $\nabla^{univ}$ over $X\times X^\vee$ 
 is of Dolbeault type (1,1) on the complex manifold $X\times X^\vee$, under every choice of complex structure in the hyperk\"ahler triple. We may call such a connection \textbf{triholomorphic}.
\end{prop}

\begin{proof}
We check its (0,2) component of $F(\nabla^{univ})$ vanishes, and the (2,0) part is similar. We test the curvature formula (\ref{universalbundlecurvature}) against complexified vectors of the shape $u+\sqrt{-1} I_i u$ and $a+\sqrt{-1}I_i a$. This is a simple calculation. For the part of curvature pairing with two tangent vectors from $T_AX^\vee\subset T_A\mathcal{A}$ (See Section \ref{hyperkahlerquotientreview} for notations),
	\[
	F(    \nabla^{univ})(a_1+ \sqrt{-1}I_ia_1, a_2+\sqrt{-1}I_i a_2)= -2 G_A \{ a_1+\sqrt{-1}I_ia_1, a_2+\sqrt{-1}I_ia_2 \},
	\]
	which vanishes by
	\[
\begin{split}
&\{ a_1+\sqrt{-1}I_ia_1, a_2+\sqrt{-1}I_ia_2 \} 
\\
& =\{ a_1, a_2\} -\{ I_i a_1, I_i a_2 \} +\sqrt{-1}(  \{  I_ia_1, a_2   \} +\{  a_1, I_ia_2     \}           )   =0.
\end{split}
\]
For the part of $F(\nabla^{univ})$ pairing with two vectors from $T_x X$, this is merely the fibrewise ASD condition. For the cross term in $F(\nabla^{univ})$, 
\[
F(    \nabla^{univ})(a+ \sqrt{-1}I_ia, u+\sqrt{-1}I_i u)=\langle a+ \sqrt{-1}I_ia, u+\sqrt{-1}I_i u \rangle,
\]
which vanishes by
$
\langle I_ia, v\rangle= -\langle a, I_iv\rangle.
$

We have tacitly used that $I_ia_j\in T_AX^\vee$ are still in Coulumb gauge, and  $I_i$ acts on the coupled 1-form $a_j$ by pointwise applying the negative of precomposition. These come from description of the hyperk\"ahler quotient construction (\cf Section \ref{hyperkahlerquotientreview}).
\end{proof}

\begin{thm}\label{universalconnectionmainproperties}
There is a Hermitian connection on $\mathcal{E}\to X\times X^\vee$, still denoted $\nabla^{univ}$, which lifts the triholomorphic $PU(r)$ connection, and whose central curvature satisfies
\begin{equation}\label{centralcurvatureBfield}
\frac{\sqrt{-1}}{2\pi r}\Tr F(\nabla^{univ})= \mathcal{B}+\mathcal{B'},  
\end{equation}	
where $	\mathcal{B},\mathcal{B'}$ are the harmonic 2-forms representing $\frac{1}{r}c_1(E)\in H^2(X)$ and $\frac{1}{r}c_1(\mathcal{E}|_x) \in H^2(X^\vee)$ for any $x\in X$. In particular, when $c_1(E)$ is orthogonal to the hyperk\"ahler triple on $X$ and $c_1(\mathcal{E}|_x)$ is orthogonal to the hyperk\"ahler triple on $X^\vee$, then $\nabla^{univ}$ is a triholomorphic Hermitian connection.
\end{thm}

\begin{proof}
To obtain $\nabla^{univ}$ it is enough to prescribe a $U(1)$ connection on the line bundle $\Lambda^r \mathcal{E}$. The line bundle $\Lambda^r \mathcal{E}$ is topologically the tensor product of line bundles $L\to X$ and $L'\to X^\vee$; on both line bundles we can find a connection with prescribed curvature $r\mathcal{B}$ and $r\mathcal{B'}$ because they represent the appropriate Chern class; taking the tensor product connection gives the connection on $\Lambda^r \mathcal{E}$. This gives us $\nabla^{univ}$.

When the first Chern classes satisfy the orthogonality conditions, then Hodge theory implies $\mathcal{B}\wedge \omega_i=0$ and $\mathcal{B}'\wedge\omega^{X^\vee}_i=0$, \ie the central part of $F(\nabla^{univ})$ is triholomorphic. But we also know the associated $PU(r)$ connection is triholomorphic, so the $U(r)$ connection $\nabla^{univ}$ is triholomorphic as well. 
\end{proof}

\begin{rmk}
The Chern class conditions are necessary for the existence of triholomorphic $U(r)$ connections, because a triholomorphic connection restricted to any fibre copy of $X$ and $X^\vee$ is ASD. A related issue is that $\nabla^{univ}$ is not always holomorphic with respect to a given complex structure on $X^\vee$, namely the universal bundle in the algebro-geometric sense needs not always exist.	
\end{rmk}

\begin{rmk}
The construction of the $U(r)$ connection $\nabla^{univ}$ has the ambiguity of twisting by a $u(1)$-valued exact 1-form pulled back from $X^\vee$, alternatively thought of as a flat $U(1)$ connection. This is intimately related to the fact that in algebraic geometry, the definition of the universal family involves a possible twist by a holomorphic line bundle.
\end{rmk}

\section{The Mukai dual of a K3 surface}\label{TheMukaidualofaK3surface}

% We show that $\nabla^{univ}$ is \textbf{triholomorphic} if and only if a Chern class condition $c_1(\mathcal{E}|_x)\cup \omega_i^{X^\vee}=0$ is satisfied. (\cf Proposition \ref{triholomorphicconnectiononMukaidual}). There exists an alternative characterisation of this condition purely in terms of the period data on $X$ (\cf Theorem \ref{Mukaidualhyperkahlertriple}). 

In this Chapter we study the geometry of the Mukai dual (\cf Definition \ref{Mukaidualdefinition}) in more detail using the $U(r)$ connection $\nabla^{univ}$ in Theorem \ref{universalconnectionmainproperties}, with particular emphasis given to the idea of \textbf{duality}. 
 We rephrase the relation of the \textbf{hyperk\"ahler periods} on $X$ and $X^\vee$ in terms of Donaldson's $\mu$-map in Section \ref{HyperkahlerperiodsontheMukaidualK3}.
The induced triholomorphic $PU(r)$ connection induces a family of ASD connections on $X^\vee$, parametrised by $X$. Modulo the issue of strict stability, \ie the irreducibility of these ASD connections, this allows us to interpret $X$ as a moduli space of $PU(r)$ ASD connections over $X^\vee$, inducing another \textbf{hyperk\"ahler structure} on $X$, which turns out to agree with the original hyperk\"ahler structure, as we discuss in Section \ref{ASDconnectionsonMukaidual}.

\subsection{Hyperk\"ahler periods on the Mukai dual K3 surface}\label{HyperkahlerperiodsontheMukaidualK3}

We use the connection provided by Theorem \ref{universalconnectionmainproperties} to give a differential geometric understanding of the Hodge structure of $X^\vee$. 

\begin{prop}(compare \cite{DonaldsonKronheimer} Page 197)
	The cohomology class of the hyperk\"ahler 2-form $\omega_i^{X^\vee}$ on the moduli space $X^\vee$ is given by the slant product
	\begin{equation}\label{Donaldsonmumaphyperkahler}
	[\omega_i^{X^\vee}]=(ch_2(\mathcal{E})- \frac{1}{2r} c_1(\mathcal{E}) ^2)\wedge [\omega_i]/[X]= -\frac{1}{2r}p_1(ad(\mathcal{E}))\cup [\omega_i]/[X].
	\end{equation}
\end{prop}

\begin{proof}
	
We can represent the Chern character $ch(\mathcal{E})$ of the universal bundle in terms of the curvature forms:
\[
	Ch(\nabla^{univ})= \Tr ( \exp \frac{ \sqrt{-1} }{2\pi} F(\nabla^{univ})  )=r+ Ch_1(\nabla^{univ}) +Ch_2(\nabla^{univ})+\ldots.
\]
In particular 
$
	Ch_2(\nabla^{univ}) =-\frac{1}{8\pi^2} \Tr (F(\nabla^{univ})\wedge F(\nabla^{univ}  ) ).
$
It is convenient to decompose the curvature $F(\nabla^{univ})$ into 3 parts, depending on whether the components of the 2-form factor comes from $X$ or $X^\vee$,
\begin{equation}\label{typedecompositionofuniversalcurvature}
 F(\nabla^{univ})=  F(\nabla^{univ})^{X, X} +F(\nabla^{univ})^{X, X^\vee}+ F(\nabla^{univ})^{X^\vee, X^\vee}.
 \end{equation}
 For later convenience, we denote \begin{equation}\label{definitionofOmega}
 \Omega=  F(\nabla^{univ})^{X, X^\vee}.
 \end{equation}

Thus the slant product $ch_2(E)\cup [\omega_i]/[X]$ can be represented by the integration along fibres
\begin{equation*}
	\begin{split}
	&\int_X Ch_2(\nabla^{univ}) \wedge \omega_i=-\frac{1}{8\pi^2}\int_X  \Tr (F(\nabla^{univ})\wedge F(\nabla^{univ} ) )\wedge \omega_i  \\
	&=-\frac{1}{8\pi^2} \int_X \Tr (\Omega\wedge \Omega)\wedge \omega_i-\frac{1}{4\pi^2}\int_X \Tr F(\nabla^{univ})^{X,X} \wedge F(\nabla^{univ} )^{X^\vee, X^\vee} \wedge \omega_i \\
	&=-\frac{1}{8\pi^2} \int_X \Tr (\Omega\wedge \Omega)\wedge \omega_i+ (\int_X \mathcal{B}\wedge \omega_i) r \mathcal{B}'
	. 
	\end{split}
\end{equation*}
which is a 2-form on $X^\vee$. The last equality here uses \[
F(\nabla^{univ})^{X,X}\wedge \omega_i= (\mathcal{B}\wedge \omega_i) I_\mathcal{E} .\] 
Now take the cohomology class, we see
\[
\begin{split}
[-\frac{1}{8\pi^2} \int_X \Tr (\Omega\wedge \Omega)\wedge \omega_i]&=\int_X (ch_2(\mathcal{E})- \frac{1}{r} c_1(\mathcal{E}|_x) c_1(E))\wedge [\omega_i] \\
&=(ch_2(\mathcal{E})- \frac{1}{2r} c_1(\mathcal{E}) ^2)\wedge [\omega_i]/[X] \\
&=\frac{1}{2r}p_1(ad(\mathcal{E}) )\cup [\omega_i]/[X].
\end{split}
\]

	We evaluate the 2-form at the point $A\in X^\vee$ on tangent vectors $a_1, a_2\in T_A X^\vee$, to get
	\[
	\begin{split}
	& -\frac{1}{8\pi^2}\iota_{a_2}\iota_{a_1} \int_X \Tr (\Omega\wedge \Omega)\wedge \omega_i \\
	=&
	-\frac{1}{8\pi^2}\iota_{a_2} \int_X \Tr (\iota_{a_1}\Omega\wedge \Omega)\wedge \omega_i
	-\frac{1}{8\pi^2}\iota_{a_2} \int_X \Tr (\Omega\wedge \iota_{a_1}\Omega)\wedge \omega_i \\
	=& -\frac{1}{4\pi^2}\iota_{a_2} \int_X \Tr (\iota_{a_1}\Omega\wedge \Omega)\wedge \omega_i \\
	=& \frac{1}{4\pi^2} \int_X \Tr (\iota_{a_1}\Omega\wedge\iota_{a_2} \Omega)\wedge \omega_i \\
	=& \frac{1}{4\pi^2} \int_X \Tr (a_1\wedge a_2)\wedge \omega_i .
	\end{split}
	\]
	The last equality uses (\ref{universalbundlecurvature}) which says in particular
	\[
	\Omega(a, v)=\langle a, v \rangle, \quad a\in T_A X^\vee, v\in T_x X.
	\]
	We recognise $\frac{1}{4\pi^2} \int_X \Tr (a_1\wedge a_2)\wedge \omega_i $ as minus the hyperk\"ahler form $\omega_i^{X^\vee}$ on the moduli space $X^\vee$, hence the claim. 
\end{proof}

\begin{rmk}
	This quite delicate computation should be compared to \cite{DonaldsonKronheimer}, Page 197, which is essentially the same calculation but outputs different numerical factors. To clarify, our convention is that integration along $X$ commutes with wedging by forms on $X^\vee$.
\end{rmk}

\begin{rmk}
When the rank $r=2$, the relation between $\omega_i^{X^\vee}$ and $\omega_i$ is Donaldson's $\mu$-map in the theory of 4-manifold invariants:
\begin{equation}\label{Donaldsonmumap}
\tilde{\mu}: H^2(X) \to H^2(X^\vee), \quad \alpha\mapsto \alpha\cup \frac{-1}{2r} p_1(ad(\mathcal{E}))\cup [\omega_i]/[X].
\end{equation}
 
\end{rmk}

\begin{thm} \label{Mukaidualhyperkahlertriple}
The $\mu$-map is an isometry. In particular
the volume of $X$ and $X^\vee$ are equal. 
\end{thm}

\begin{proof}
We shall make use of FM transform on cohomology in Section \ref{Mukaidualreview}.
We can compute the Mukai vector $v(\mathcal{E})= ch(\mathcal{E})\sqrt{Td(X) Td(X^\vee)}$, and $v(\mathcal{E}^\vee)= ch(\mathcal{E}^\vee)\sqrt{Td(X) Td(X^\vee)}$. Since $X$ and $X^\vee$ are K3 surfaces, we know $Td(X)=1+2[X]^*$ and $Td(X^\vee)=1+ 2[X^\vee]^*$. 
	Thus for any $\alpha\in H^2(X)$,
	\[
	\begin{split}
	 FM(\alpha) _{H^0(X^\vee)}&=- \int_X \alpha \cup c_1(E), \\
	 FM(\alpha) _{H^2(X^\vee)}&=  \alpha\cup ch_2(\mathcal{E^\vee})/[X]=-\tilde{\mu}(\alpha)  +\frac{1}{r}(\int_X c_1(E)\wedge \alpha)c_1(\mathcal{E}|_x)
	.
	\end{split}
	\]
Here $\tilde{\mu}$ is the Donaldson $\mu$-map.	To calculate $FM(\alpha)_{H^4}$, we observe the FM transform of the fundamental class of $X$ is
\[
FM([X]^*)= v(\mathcal{E}^\vee)\cup [X]^*/[X]= v(\mathcal{E}^\vee|_x) \in H^*(X^\vee), \quad \forall x\in X
\]
so the Mukai pairing
\[
\begin{split}
0&=-([X]^*, \alpha)=-(FM([X]^*), FM(\alpha) )\\
&=FM(\alpha)_{H^0} (\int_{X^\vee} ch_2(\mathcal{E}|_x)+r)+\int_{X^\vee} FM(\alpha)_{H^2}\cup c_1(\mathcal{E}|_x)+ r \int_{X^\vee} FM(\alpha)_{H^4}\\
&=FM(\alpha)_{H^0} \frac{1}{2r}(\int_{X^\vee} c_1(\mathcal{E}|_x)^2)+\int_{X^\vee} FM(\alpha)_{H^2}\cup c_1(\mathcal{E}|_x)+ r \int_{X^\vee} FM(\alpha)_{H^4}\\
&= r \int_{X^\vee} FM(\alpha)_{H^4}- \int_{X^\vee}\tilde{\mu}(\alpha)\cup c_1(\mathcal{E}|_x) + \frac{1}{2r} (\int_X\alpha\cup c_1(E)) \int_{X^\vee} c_1(\mathcal{E}|_x)^2,
\end{split}
\]
hence
\[
FM(\alpha)_{H^4}=\frac{1}{r} \tilde{\mu}(\alpha)\cup c_1(\mathcal{E}|_x) -\frac{1}{2r^2}(\int_X\alpha \cup c_1(E)) c_1(\mathcal{E}|_x)^2.
\]

The $\mu$-map is an isometry because 
\[
\begin{split}
\int_X \alpha^2&=
(\alpha, \alpha)= (FM(\alpha), FM(\alpha) )\\
&= FM(\alpha)_{H^2}^2- 2 FM(\alpha)_{H^0} FM(\alpha)_{H^4}= \int_{X^\vee} \tilde{\mu}(\alpha)^2,
\end{split}
\]	
where the last step is proved by substituting the expressions for components of $FM(\alpha)$ and cancelling out terms.

Since
$\omega_i$ and $\omega_i^{X^\vee}$ are related by the $\mu$-map, their volumes are equal.
\end{proof}

\begin{rmk}
When $r=2$, the integral $q(\alpha)=\int_{X^\vee} \tilde{\mu}(\alpha)^2$
 is one of  Donaldson's polynomial invariants for K3 surfaces, and this volume computation is just the higher rank generalisation.
\end{rmk}

\subsection{ASD connections on $X^\vee$ and hyperk\"ahler structure}\label{ASDconnectionsonMukaidual}

The universal $U(r)$ connection $\nabla^{univ}$ on $\mathcal{E}\to X\times X^\vee$ in Theorem \ref{universalconnectionmainproperties} induces a family of HYM connections (or equivalently, ASD $PU(r)$ connections) on  $E'\simeq \mathcal{E}^\vee|_x\to X^\vee$ parametrised by $x\in X$, where $E'\to X^\vee$ denotes the underlying Hermitian bundle. The reason for considering the dual bundle is to be compatible with the inverse Fourier-Mukai transforms (\cf Section \ref{Mukaidualreview}). The goal of this Section is

\begin{thm}\label{doubledualityhyperkahlerforms}
If all these HYM connections are irreducible, then $X$ is the moduli space $\mathcal{M}^s(v(E')))$, and the hyperk\"ahler structure on $X$ induced from the moduli interpretation agrees with $(g, \omega_i)$. 
\end{thm}

\begin{rmk}
This means $X$ is also the Mukai dual of $X^\vee$, so $X$ and $X^\vee$ appear on equal footing. We will in fact separate the irreducibility/stability assumption from most of the intermediate arguments.	
\end{rmk}

 We first observe that $X$ has the correct virtual dimension:

\begin{lem}
The Mukai vector of $E'\to X^\vee$ satisfies $(v(E'), v(E'))=0$, so the moduli space $\mathcal{M}^s(v(E'))$ has real virtual dimension $4$.
\end{lem}

\begin{proof}
The Fourier-Mukai transform on cohomology (\cf Section \ref{Mukaidualreview}) maps the fundamental class $[X]^*\in H^4(X)$ to 
\[
FM( [X]^*  )= [X]^*\cup v(\mathcal{E}^\vee) /[X] = v(E') \in H^*(X^\vee).
\]
Now because the Fourier-Mukai transform preserves the Mukai pairing, we have
\[
( v(E') , v(E')  )=( [X]^*, [X]^*     )   =     0.
\]
The formula for the complex virtual dimension is $(v(E'), v(E') )+2=2$. (\cf Section \ref{Mukaidualreview}). Hence the claim.
\end{proof}

We study the variation of these ASD $PU(r)$ connections on $X^\vee$ as $x\in X$ varies. Let $T'\subset X$ be an open set with coordinates $x_i$, containing a point of interest $x_0\in T'$. We can topologically identify the smooth bundles $\mathcal{E^\vee}|_{x}\to \{x\}\times X^\vee$ for $x\in T'$, as the bundle $E'\to X^\vee$. Then we get a varying family of ASD connections $A_{x}= -\nabla^{univ,t}|_{x}$ parametrised by $x\in T'$, inside an infinite dimensional space modelled on $\Omega^1(X^\vee, ad_0 (E'))$ (\cf Section \ref{hyperkahlerquotientreview}). The minus transpose takes place because we are working with the dual bundle.

The topological identification  can be twisted by gauge transformations. To rigidify the situation, we may use the transposed universal connection $-\nabla^{univ,t}$ to give an infinitesimal trivialisation around the central fibre, so that the derivative $\frac{\partial A_{x}}{\partial x_i} $ at the point $x_0$ agrees with the commutator 
\[
[-\nabla^{univ,t}_{ \frac{\partial }{\partial x_i}     } , -\nabla^{univ,t}  ]= -\iota_{\frac{\partial }{\partial x_i}  } F(\nabla^{univ,t})= -\iota_{\frac{\partial }{\partial x_i}  } \Omega^t\in \Omega^1(X^\vee, ad_0 (E')), \]
where we have restricted the coupled 1-form $-\iota_{\frac{\partial }{\partial x_i}  } F(\nabla^{univ,t})  $ to the central fibre, and only the component $\Omega$ of $F(\nabla^{univ})$ actually contributes. We can regard this commutator as the \textbf{infinitesimal variation of the ASD connections}.

\begin{lem}\label{infinitesimalvariationofASDconnectionisinCoulumbgauge}
The infinitesimal variation $  a_i=-\iota_{\frac{\partial }{\partial x_i}  } \Omega^t \in \Omega^1(X^\vee, ad_0 (E')) $ satisfies the linearised ASD equation $d_{A_{x_0}}^+ a_i=0$, and the Coulumb gauge fixing condition $d_{A_{x_0}}^* a_i=0$.
\end{lem}

\begin{proof}
Since $a_i$ is traceless, the $PU(r)$ version and the $U(r)$ version of these equations are equivalent. 
The linearised ASD equation follows directly from variation of ASD connections. To show the Coulumb gauge condition, we start from the equation over $X\times X^\vee$,
\begin{equation}
d_{A^{univ}}^* F(\nabla^{univ})=0, 
\end{equation}
which is a consequence of the triholomorphic property. Now recall the decomposition (\ref{typedecompositionofuniversalcurvature}) and (\ref{definitionofOmega}), 
\[
F(\nabla^{univ})=  F(\nabla^{univ})^{X, X} +\Omega+ F(\nabla^{univ})^{X^\vee, X^\vee}.
\]
Similarly, the operator $d_{A^{univ}}^ * $ can be decomposed into 
$
d_{A^{univ}}^ *= d_{A^{univ}}^ {X, *}+d_{A^{univ}}^ { X^\vee, *},
$
corresponding to decreasing the bidegree of forms in the $X$ direction or the $X^\vee$ direction. The Yang-Mills equation decomposes as
\[
d_{A^{univ}}^ {X, *} F(\nabla^{univ})^{X, X}+ d_{A^{univ}}^ { X^\vee, *} \Omega=0,
\]
and 
\[
d_{A^{univ}}^ { X^\vee, *} F(\nabla^{univ})^{X^\vee, X^\vee}+ d_{A^{univ}}^ { X, *} \Omega=0.
\]
We observe that the connection is ASD on each fibre, so the Yang-Mills equation is satisfied, which implies \[d_{A^{univ}}^ { X, *} F(\nabla^{univ})^{X, X}=0, \quad  d_{A^{univ}}^ { X^\vee, *} F(\nabla^{univ})^{X^\vee, X^\vee}=0.\]
Thus in turn we have
\begin{equation}\label{YangMillsoncomponents}
d_{A^{univ}}^ { X^\vee, *} \Omega=0, \quad d_{A^{univ}}^ { X, *}\Omega=0.
\end{equation}
Hence on the central fibre,
\[
d_{A^{univ}}^ { X^\vee, *} \iota_{ \frac{\partial }{\partial x_i}  } \Omega
 = - \iota_{ \frac{\partial }{\partial x_i}  }   d_{A^{univ}}^ { X^\vee, *}\Omega=0 \in \Omega^0(X^\vee, ad_0(E')), 
\]
which is the desired Coulumb gauge condition, after taking transpose.
\end{proof}

\begin{rmk}
It is curious that the Coulumb gauge condition on the fibre is related to the Yang-Mills condition on the total space.
\end{rmk}

The infinitesimal variation of ASD connections behaves well under quaternionic actions:

\begin{lem}\label{infinitesimalvariationofASDconnectioniscompatiblewithquaternionicaction}
For any choice of complex structure $I_k$, The infinitesimal variation induced by $I_k \frac{\partial}{\partial x_i}$ is $I_k a_i$, where by $I_k a_i$ we mean the negative of precomposition by $I_k$.
\end{lem}

\begin{proof}
We compute using the triholomorphic property of $\Omega$ as follows:
\[
\begin{split}
\iota_{ I_k \frac{\partial}{\partial x_i} } \Omega =\Omega ( I_k \frac{\partial}{\partial x_i} , \cdot{}   )=\Omega( \frac{\partial}{\partial x_i}, -I_k(\_)          )
= - \iota_{  \frac{\partial}{\partial x_i} } \Omega \circ I_k.
\end{split}
\]
We then take the transpose to see the result.
\end{proof}

We interprete the family of ASD instantons on $E'\to X^\vee$ as giving a map from $X$ to the space of $PU(r)$ connections on $E'\to X^\vee$, modulo gauge equivalence classes. We can compare this with the the hyperk\"ahler quotient construction in Section \ref{hyperkahlerquotientreview}. The optimal hope is to identify $X$ with the moduli space $\mathcal{M}^s(v(E'))$ of ASD instantons on $E'\to X^\vee$ with its hyperk\"ahler structure. There are several issues to this. First, in general these ASD instantons may not be irreducible, to ensure the smoothness of moduli space. Second, it is not a priori clear that the map from $X$ to the moduli space is an immersion, because infinitesimal variations may be zero. 

Despite these issues, we can nevertheless write down the (semi)-metric and the triple of 2-forms on $X$ induced from the moduli space picture, since Lemma \ref{infinitesimalvariationofASDconnectionisinCoulumbgauge} already puts us in the appropriate gauge fixing condition. (Compare with Section \ref{hyperkahlerquotientreview}). The (semi)-metric is given by 
\begin{equation}
g^{\vee\vee}(\frac{\partial}{\partial x_i},\frac{\partial}{\partial x_j} )=\frac{1}{4\pi^2}\int_{X^\vee} \langle \iota_{\frac{\partial}{\partial x_i}} \Omega^t , \iota_{\frac{\partial}{\partial x_j}} \Omega^t \rangle d\text{Vol}_{X^\vee},
\end{equation}
where the pointwise inner product of coupled 1-forms on $X^\vee$ is given by combining the negative of the trace pairing on the bundle part, and the inner product on the 1-form part. This is clearly smooth and semi-positive definite, but not a priori known to be positive definite. The corresponding triple of 2-forms $\omega^{\vee\vee}_i$ are defined by the requirement
\[
\omega_k^{\vee\vee}(\_, \_)= g^{\vee\vee}( I_k\_, \_  ).
\]
More explicitly, we can write down
\[
\begin{split}
\omega_k^{\vee\vee}(\frac{\partial}{\partial x_i},\frac{\partial}{\partial x_j} )&=\frac{-1}{4\pi^2}\int_{X^\vee} \Tr ( \iota_{\frac{\partial}{\partial x_i}} F(\nabla^{univ,t}) \wedge \iota_{\frac{\partial}{\partial x_j}} F(\nabla^{univ,t}) ) \wedge \omega_k^{X^\vee}  \\
& =\frac{-1}{4\pi^2}\int_{X^\vee} \Tr ( \iota_{\frac{\partial}{\partial x_i}} \Omega \wedge \iota_{\frac{\partial}{\partial x_j}} \Omega )\wedge \omega_k^{X^\vee} 
.
\end{split}
\]
We can write more concisely by wedging $dx_i\wedge dx_j$ and sum up:
\begin{equation}\label{hyperkahlerformdoubledual}
\omega_k^{\vee\vee}=\frac{1}{8\pi^2}\int_{X^\vee} \Tr (  F(\nabla^{univ,t}) \wedge  F(\nabla^{univ,t}) ) \wedge \omega_k^{X^\vee}. 
\end{equation}
Here we pick up another minus sign when we commute $dx_j$ with the coupled 1-form $\iota_{\frac{\partial}{\partial x_i}} F(\nabla^{univ,t}) $, and an extra factor $\frac{1}{2}$ because \[\omega_k^{\vee\vee}=\frac{1}{2}\sum_{i,j}\omega_k^{\vee\vee}(\frac{\partial}{\partial x_i},\frac{\partial}{\partial x_j}) dx_i\wedge dx_j  .\]

\begin{lem}
The triple of 2-forms $\omega^{\vee\vee}_i$ are closed.	
\end{lem}

\begin{proof}
We notice the trace of curvature terms are closed on $X\times X^\vee$ by Chern-Weil theory. Thus the integrand $\Tr (  F(\nabla^{univ,t}) \wedge  F(\nabla^{univ,t}) ) \wedge \omega_i^{X^\vee}$ is closed on $X\times X^\vee$, and the integration on fibres preserves this closedness.
\end{proof}

\begin{lem}
The hyperk\"ahler structure $(g^{\vee\vee}, \omega_i^{\vee\vee})$ agrees with $(g, \omega_i)$.
\end{lem}

\begin{proof}
Both $g$ and $g^{\vee\vee}$ are compatible with the quaternionic actions of $I_1$, $I_2$, $I_3$, using Lemma \ref{infinitesimalvariationofASDconnectioniscompatiblewithquaternionicaction}. Such metrics are determined up to a scalar function $f$, so $g^{\vee\vee}=fg$ is conformal to $g$. Now $\omega_i^{\vee\vee}=f\omega_i$ is closed, so $df\wedge \omega_i=0$, which implies $df=0$, hence $f$ is a constant.

To pin down the constant requires some input from cohomology. By a computation completely analogous to (\ref{Donaldsonmumap}), 
\[
[\omega_i^{\vee\vee}]= -\frac{1}{2r} p_1(ad(\mathcal{E}))\cup [\omega_i^{X^\vee}]/[X^\vee]
\]
is given by the adjoint of Donaldson's $\mu$-map $\tilde{\mu}$ acting on $[\omega_i^{X^\vee}]$. Since $\tilde{\mu}$ is an isometry, its adjoint is its inverse, so $[\omega_i^{X^\vee}]=\tilde{\mu}^{-1}([\omega_i^{X^\vee}])=[\omega_i]$, hence the constant $f=1$.
\end{proof}

We now prove Theorem \ref{doubledualityhyperkahlerforms}.

\begin{proof}
Since $g^{\vee\vee}$ is positive definite, we see a posteriori that the infinitesimal variation is nowhere vanishing, so the tautological map $X\to \mathcal{M}^s(v(E'))$ is an immersion. Since $X$ has the right virtual dimension, this is a local isomorphism once we assume irreducibility/stability. Thus $\mathcal{M}^s(v(E'))$ contains a compact component, namely the image of $X$, so $\mathcal{M}^s(v(E'))$ is the image by Theorem \ref{Mukaidualalgebraic}. But $\mathcal{M}^s(v(E'))$ is a hyperk\"ahler surface, so cannot be an unramified quotient of $X$. We see $X=\mathcal{M}^s(v(E'))$.

The hyperk\"ahler structure induced by the moduli interpretation is precisely $(g^{\vee\vee}, \omega_i^{\vee\vee})$, so agrees with $(g, \omega_i)$.
\end{proof}

\section{The Nahm transform on K3 surfaces}\label{TheNahmtransformonK3surfaces}

The aim of this Chapter is to develop the theory of the Nahm transform (a.k.a the Fourier-Mukai transform) on K3 surfaces in the differential geometric setting, by adapting the work of Braam and Baal \cite{BraamBaal} on the Nahm transform over the 4-torus. There are two major technical difficulties: the non-flat nature of the ambient metric, and the non-explicit nature of the universal connection.

Section \ref{Nahmtransform} defines the \textbf{Nahm transform} for irredubible HYM connections on bundles  of certain slopes over the hyperk\"ahler K3 surface, gives a brief comparison with the algebraic viewpoint, and then proceed to use spinor methods to show that the \textbf{Hermitian Yang-Mills} property is preserved by the Nahm transform. To prove this, we use a formula of the curvature of the transformed connection in terms of certain Green operators, as in \cite{BraamBaal}, and the heart of the argument is to introduce the natural action of complex structure operators on positive spinors (\cf the Appendix), which commute with the Green operators. This allows us to reduce the proof to a local calculation, where the key input is the properties of the universal connection on $\mathcal{E}\to X\times X^\vee$ from Theorem \ref{universalconnectionmainproperties}. 

The rest of the Chapter aims to prove the \textbf{Fourier inversion theorem} (\cf \ref{Fourierinversiontheorem}) for the Nahm transform by differential geometric calculations, and is by far the most difficult part of this paper. This means the Nahm transform suitably applied twice (the `\textbf{inverse Nahm transform}') would give back the original bundle with its HYM connection. A technical feature is that a large amount of cancellation effects take place to remove many complicated expressions.

Section \ref{inverseNahmtransformcomparisonmap} defines a \textbf{canonical comparison map} between the original bundle and the inverse Nahm transform, built out of Green operators and curvature operators on coupled spinors. We then show this is well defined, namely that the map lands in the correct target, by verifying a coupled Dirac equation.

Section \ref{Injectiveisometry} shows that the canonical comparison map is a  \textbf{Hermitian isometry}. The Hermitian inner product on the inverse Nahm transform takes the form of a correlator type expression, which is converted via functional analytic calculations into the integral of a Laplacian type expression involving some singular Schwartz kernels. The proof then proceeds by deriving delicate asymptotic formulae of the singularity, and evaluating them in the limit.

Section \ref{Comparingtheconnections} \textbf{compares connections} on the original bundle with the inverse Nahm transform, and show that they agree under the canonical comparison map. The proof extends the asymptotic calculation in Section \ref{Injectiveisometry} to higher order.

\subsection{The Nahm transform}\label{Nahmtransform}

The aim of this Section is to translate the Fourier-Mukai transform for stable vector bundles $\mathcal{F}$ with the same slope as $E$ on the K3 surface $X$, 
into the language of the Nahm transform, in analogy with the work of Braam and Baal \cite{BraamBaal} in the 4-torus case. The main result says under certain conditions the HYM condition is preserved under Nahm transform, proved using spinor techniques.

We briefly review the \textbf{algebraic viewpoint} for completeness, which will not be substantially used. Let $(X, g, \omega_1, \omega_2, \omega_3)$ be a hyperk\"ahler K3 surface, and $X^\vee$ a Mukai dual K3 surface, such that there is a universal bundle $\mathcal{E}\to X\times X^\vee$ with a universal connection $\nabla^{univ}$, as in Theorem \ref{universalconnectionmainproperties}. Let $\mathcal{F}$ be a stable holomorphic vector bundle with the same slope as $E$, such that
 $F\not \simeq \mathcal{E}|_\tau$ for any $\tau\in X^\vee$, which by the properties of stable vector bundles implies $\text{H}^0(X,\underline\Hom ( \mathcal{E}|_\tau, \mathcal{F}) )=0 $ and 
 $\text{H}^0(X, \underline\Hom (\mathcal{F}, \mathcal{E}|_\tau))=\text{Ext}^0(\mathcal{F},\mathcal{E}|_\tau)=0$. By Serre duality, we also have $\text{Ext}^2(\mathcal{E}|_{\tau}, \mathcal{F})\simeq \text{Ext}^0( \mathcal{F}, \mathcal{E}|_\tau)^*=0$. The Fourier-Mukai transform on sheaves produces a vector bundle $R^1pr_{X^\vee *}( \mathcal{E}^\vee\otimes pr^*_{X}\mathcal{F}       )$ over $X^\vee$, whose fibres at $\tau \in X^\vee$ are $ \text{H}^1(X, \underline{\Hom}(\mathcal{E}|_{\tau}, \mathcal{F}  ) )\simeq \text{Ext}^1( \mathcal{E}|_{\tau}  , \mathcal{F}   )  $ by the base change theorem. Here $\underline{\Hom}$ denotes the holomorphic $\Hom$ bundle.

From the \textbf{differential geometric viewpoint}, the Fourier-Mukai transform can be understood in terms of the Nahm transform  (Compare \cite{BraamBaal}, Section 1). We think of the stable holomorphic bundle $\mathcal{F}$ equivalently as a vector bundle $\mathcal{F}$ with an irreducible HYM connection $ \alpha$. (This notation is to avoid confusion with connections on $\mathcal{E}$.) Another description which does not favour any complex structure, is that $\alpha$ induces an ASD $PU(rk(\mathcal{F}))$ connection, and the central curvature satisfies
\[
\frac{\sqrt{-1}}{ 2\pi} \Tr F_\alpha=rk(\mathcal{F}) \mathcal{B}_\mathcal{F},
\]
where $\mathcal{B}_\mathcal{F}$ is the harmonic 2-form representing the class $\frac{1}{rk(\mathcal{F})} c_1(\mathcal{F})$. The slope condition is equivalent to 
\[
(\mathcal{B}_\mathcal{F}-\mathcal{B})\wedge \omega_i=0,
\]
where $\mathcal{B}$ is the harmonic representative of $\frac{1}{rk(E)} c_1(E)$ as in Theorem \ref{universalconnectionmainproperties}.

 Let $S^+_X, S^{-} _X\to X$ be the spinor bundles over $X$. This allows us to define the Dirac operators, by coupling $\alpha$ to the universal connection $\nabla^{univ}$ and the Levi-Civita connection on spinors:
\[
\begin{split}
& D^+_{\alpha_\tau  }: \Gamma(X, S^+_X \otimes\Hom( \mathcal{E}|_{\tau}, \mathcal{F}   )  ) \to \Gamma(X, S^-_X \otimes\Hom( \mathcal{E}|_{\tau}, \mathcal{F}   )  ), \\
& D^-_{ \alpha_\tau  }: \Gamma(X, S^-_X \otimes\Hom( \mathcal{E}|_{\tau}, \mathcal{F}   )  ) \to \Gamma(X, S^+_X \otimes\Hom( \mathcal{E}|_{\tau}, \mathcal{F}   )  ).
\end{split}
\]
These two operators are formally adjoint to each other. It is important that the coupled connection $\alpha_\tau$ on $\text{Hom}(\mathcal{E}|_\tau, \mathcal{F})$ is an ASD unitary connection, instead of merely ASD projective unitary connection.

\begin{rmk}
Using the arguments in \cite{DonaldsonKronheimer}  Section 3.2, the Dirac operator is closely related to the Dolbeault complex, and its kernel can be identified with the Dolbeault cohomology groups:
 \[
 \ker D^+_{\alpha_\tau  }\simeq \text{H}^0(X, \underline{\Hom} (\mathcal{E}|_\tau, \mathcal{F}    )    ) \oplus  \text{H}^2(X, \underline{\Hom} (\mathcal{E}|_\tau, \mathcal{F}    )    ) ,  \]
\[
\ker D^-_{\alpha_\tau  }\simeq \text{H}^1(X, \underline{\Hom} (\mathcal{E}|_\tau, \mathcal{F}    )    ) .
\]

\end{rmk}

\begin{lem}\label{vanishingkernellemma}
If the topological type of $\mathcal{F}$ is different from $E$, then $\ker D^+_{\alpha_\tau  } \subset \Gamma(X, S^+_X \otimes\Hom( \mathcal{E}|_\tau, \mathcal{F} )   )$ is zero for all $\tau\in X^\vee$.
\end{lem}

\begin{proof}\label{Lichnerowichformulaargument}
By the Lichnerowich formula
\[
 D^-_{ \alpha_\tau  } D^+_{\alpha_\tau  }= \nabla^*_{\alpha_\tau} \nabla_{\alpha_\tau} + F_{\alpha_\tau}^+ + \frac{1}{4}R,
\]
where $ F_{\alpha_\tau}^+$ denotes the action of the curvature tensor on positive spinors, and $R$ is the scalar curvature. On a hyperk\"ahler K3 surface, $R$ is zero, and the ASD condition means $F_{\alpha_\tau}^+=0$. By a Bochner type argument, we see that $\ker D^+_{\alpha_\tau  }$ consists entirely of parallel coupled spinor fields. For a K3 surface $X$, the positive spin bundle $S^+_X$ is covariantly trivial. So using $\text{H}^0(X,\underline\Hom( \mathcal{E}|_{\tau}, \mathcal{F} ) ) =0$, this kernel must vanish.
\end{proof}

The vector spaces $\hat{\mathcal{F} }|_\tau=\ker D^-_{ \alpha_\tau  }$ therefore fit together into a vector bundle $\hat{\mathcal{F}}\to X^\vee$. This can be thought of as a subbundle of the infinite dimensional Hermitian vector bundle $\hat{H}\to X^\vee$, whose fibres are $\hat{H}_\tau= \Gamma( X, S^-_X \otimes \Hom( \mathcal{E}|_{\tau}, \mathcal{F}   )   )$. This bundle $\hat{H}$ has a natural covariant derivative $\hat{d}$, induced by the universal connection on $\mathcal{E}$ in the $X^\vee$ direction. Let $\hat{\alpha}$ be the subbundle connection on $\hat{\mathcal{F}}$, more concretely described by the covariant derivative $\hat{\nabla}=P\hat{d}$, where $P: \hat{H}\to \hat{\mathcal{F}}$ is the $L^2$ projection. Using Hodge theory, the $L^2$ projection operator is expressed as
\[
P=1- D_{\alpha_\tau}^+ G_\tau D^-_{\alpha_\tau}
\]
with $G_\tau= ( D^-_{\alpha_\tau} D_{\alpha_\tau}^+   )^{-1}= ( \nabla^*_{\alpha_\tau}\nabla_{\alpha_\tau}  )^{-1}
$ acting on $\Gamma( X, \Hom ( \mathcal{E}|_\tau, \mathcal{F}   ) \otimes S^+_X  )   $.

\begin{Def}
The \textbf{Nahm transform} of $(\mathcal{F}, \alpha)$ is the pair of vector bundle with connection $(\hat{\mathcal{F}}, \hat{\alpha})$.
\end{Def}

%\begin{rmk}
%One can compute the Chern character of $\mathcal{F}$ using the  Grothendieck-Riemann-Roch theorem, by comparing with the Dolbeault description above, or apply the Atiyah-Singer theorem for families. (\cf also Remark \ref{FourierMukaicohomology}).
%\end{rmk}

We follow \cite{BraamBaal} with some  modifications to show
\begin{thm}\label{NahmtransformperservesASD}
The Nahm transform $(\hat{\mathcal{F}}, \hat{\alpha})$ is also a HYM connection, with the same slope as $E'\to X$ defined in Section \ref{HyperkahlerperiodsontheMukaidualK3}.
\end{thm}

\begin{proof}
Let $\hat{f}^j(\tau)= \psi^j_\tau (x)   \in \Gamma( X, S^-_X \otimes \Hom( \mathcal{E}|_{\tau}, \mathcal{F}   ) ) $ with $\tau\in X^\vee$ and $j=1,2\ldots , rk(\hat{\mathcal{F}})$ be a local orthonormal framing of $\hat{\mathcal{F}}\to X^\vee$. For a section $\hat{s}(\tau)= \sum_j \hat{s}_j \hat{f}^j(\tau)$, with $\hat{s}_j$ being local  $C^\infty$ functions on $X^\vee$, one can compute \[
\hat{\nabla} \hat{s}
= P\hat{d} \hat{s} =(1-D_{\alpha_\tau}^+ G_\tau D^-_{\alpha_\tau}   ) [ \hat{d} \sum \hat{s}_j (\tau) \psi^j_\tau(x)  ],
\]
In components, we can write \[
\hat{\nabla} \hat{s}= ( d \hat{s}_j + \hat{\alpha}_{jk} \hat{s}_k       ) \hat{f}^j,
\]
where the connection matrix $\hat{\alpha}_{jk}$ of $\alpha$ is equal to
\[
\hat{\alpha}_{jk}= \langle \hat{f}^j, \hat{d} \hat{f}^k\rangle=  \langle \psi^j_\tau, \hat{d} \psi^k_\tau   \rangle.
\]
The curvature matrix is
\[
\hat{F}_{ij}= d\hat{\alpha}_{ij} +\sum_k\hat{\alpha}_{ik} \wedge \hat{\alpha}_{kj}=
 \langle \hat{d} \psi^i_\tau, \wedge \hat{d} \psi^j_\tau   \rangle+ \langle  \psi^i_\tau, \wedge \hat{d}^2 \psi^j_\tau   \rangle
 + 
 \sum_k \langle \psi^i_\tau, \hat{d} \psi^k_\tau   \rangle \wedge \langle \psi^k_\tau, \hat{d} \psi^j_\tau   \rangle.
\]
Now $\langle \hat{d} \psi^i_\tau,  \psi^k_\tau   \rangle=-\langle \psi^i_\tau, \hat{d} \psi^k_\tau   \rangle$ by the compatibility with the Hermitian structures, so the third term above is recognized as $- 
\langle P \hat{d} \psi^i_\tau,\wedge \hat{d} \psi^j_\tau   \rangle$, and
\[
\begin{split}
\hat{F}_{ij}= & \langle \hat{d} \psi^i_\tau, \wedge \hat{d} \psi^j_\tau   \rangle- 
\langle P \hat{d} \psi^i_\tau,\wedge \hat{d} \psi^j_\tau   \rangle + \langle  \psi^i_\tau, \wedge \hat{d}^2 \psi^j_\tau   \rangle\\
= & \langle  D_{\alpha_\tau}^+ G_\tau D^-_{\alpha_\tau}      \hat{d} \psi^i_\tau, \wedge \hat{d} \psi^j_\tau   \rangle + \langle  \psi^i_\tau, \wedge \hat{d}^2 \psi^j_\tau   \rangle\\
= & \langle   G_\tau D^-_{\alpha_\tau}      \hat{d} \psi^i_\tau, \wedge D_{\alpha_\tau}^- \hat{d} \psi^j_\tau   \rangle+ \langle  \psi^i_\tau, \wedge \hat{d}^2 \psi^j_\tau   \rangle.
\end{split}
\]
But since $ D^-_{\alpha_\tau}\psi^i_\tau=0$,
\[ D^-_{\alpha_\tau}      \hat{d} \psi^i_\tau=[  D^-_{\alpha_\tau},       \hat{d}   ]\psi^i_\tau.
\]
Computing this commutator $[  D^-_{\alpha_\tau},       \hat{d}   ]$  requires extra care compared to the $T^4$ case in \cite{BraamBaal}. In coordinates, this is given by
\[
[  D^-_{\alpha_\tau},       \hat{d}   ]=   [ \sum_{\mu} c(dx_\mu) \nabla^{univ, t}_{\frac{\partial }{\partial x_\mu}}, \sum_\nu d\tau_\nu  \nabla^{univ,t}_{ \frac{\partial }{\partial \tau_\nu}}         ].
\]
Here $c(dx_\mu)$ means the Clifford multiplication action of the 1-form $dx_\mu$ on spinors, and the covariant derivatives are essentially only acting on the dual of the universal bundle $\mathcal{E}$, not on $\mathcal{F}$, nor the spinor factor, which is why we see the transposed universal connection.

We now recognize that 
\[ \Omega= \sum_{\mu, \nu} [  \nabla^{univ}_{\frac{\partial }{\partial x_\mu}},  \nabla^{univ}_{ \frac{\partial }{\partial \tau_\nu}}         ] dx_\mu \wedge d\tau_\nu
\]
is part of the curvature of the universal connection on  $\mathcal{E}\to   X\times X^\vee$. We can now write
\[
 D^-_{\alpha_\tau}      \hat{d} \psi^i_\tau=[  D^-_{\alpha_\tau},       \hat{d}   ]\psi^i_\tau=  -\Omega^t\cdot{} \psi^i_\tau,
\]
where by writing $-\Omega^t\cdot{} \psi^i_\tau$ we need to make cotangent vectors $dx_\mu$ on $X$ act on spinors by Clifford multiplication, take care of minus transpose on the bundle factor, and leave $d\tau_\nu$ untouched. From this computation, we see the \textbf{curvature matrix} of $\hat{\alpha}$ is
\begin{equation}\label{curvaturematrixNahmtransform}
\hat{F}_{ij}
= \langle   G_\tau \Omega^t\cdot{} \psi^i_\tau, \wedge \Omega^t\cdot{} \psi^j_\tau   \rangle+ \langle  \psi^i_\tau, \wedge \hat{d}^2 \psi^j_\tau   \rangle.
\end{equation}

We claim that this curvature matrix (\ref{curvaturematrixNahmtransform}) is HYM on $X^\vee$. This is more delicate than the flat 4-forus case in \cite{BraamBaal}. We first observe that $\hat{d}^2$ comes from the curvature of $\nabla^{univ}$ in the $X^\vee$ direction, so this contribution to $\hat{F}_{ij}$ is HYM. Its contribution to $\frac{\sqrt{-1}}{2\pi}\Tr \hat{F}$ is $-rk(\hat{\mathcal{F}})\mathcal{B}'$, with $\mathcal{B}'$ as in Theorem \ref{universalconnectionmainproperties}; the reason for the minus sign is that we are using the dualised bundle $\mathcal{E}^\vee$.

It remains to show that the term  $\langle   G_\tau \Omega^t\cdot{} \psi^i_\tau, \wedge \Omega^t\cdot{} \psi^j_\tau   \rangle$ is ASD. Equivalently, for any complex structure $I_k$, we need to show $\langle   G_\tau \Omega^t\cdot{} \psi^i_\tau, \wedge \Omega^t\cdot{} \psi^j_\tau   \rangle$ is equal to
\[
\begin{split}
 \sum \langle   G_\tau \Omega^t(\frac{\partial}{\partial \tau_a}, \frac{\partial}{\partial x_\mu}) c(dx_\mu)   \cdot{} \psi^i_\tau,  \Omega^t(\frac{\partial}{\partial \tau_b}, \frac{\partial}{\partial x_\nu}) c(dx_\nu) \cdot{} \psi^j_\tau   \rangle I_kd\tau_a \wedge I_k d\tau_b
 \\
 = 
\sum \langle   G_\tau \Omega^t(I_k \frac{\partial}{\partial \tau_a}, \frac{\partial}{\partial x_\mu}) c(dx_\mu)   \cdot{} \psi^i_\tau,  \Omega^t(I_k\frac{\partial}{\partial \tau_b}, \frac{\partial}{\partial x_\nu}) c(dx_\nu) \cdot{} \psi^j_\tau   \rangle d\tau_a \wedge  d\tau_b. 
\end{split}
\]
Using the triholomorphic property of $\Omega$, 
\[
\Omega( I_k v, w   )= -\Omega( v, I_k w ), \quad v\in TX^\vee, w\in TX,
\]
the above is
\[
\sum \langle   G_\tau \Omega^t( \frac{\partial}{\partial \tau_a}, \frac{\partial}{\partial x_\mu}) c(I_k dx_\mu)   \cdot{} \psi^i_\tau,  \Omega^t(\frac{\partial}{\partial \tau_b}, \frac{\partial}{\partial x_\nu}) c(I_kdx_\nu) \cdot{} \psi^j_\tau   \rangle d\tau_a \wedge  d\tau_b. 
\]

A source of difficulty is that, because the metric is not flat, we cannot directly commute the Green operator with Clifford multiplication $c(dx_\mu)$. The remedy is quite subtle, and we need to recall some \textbf{spin geometry} in dimension 4 (\cf the Appendix). Since $X$ is hyperk\"ahler, the complex structures $I_1$, $I_2$, $I_3$ can be made to act on $S^+_X$, and act trivially on $S^-_X$, in a way compatible with the Clifford multiplication (see (\ref{complexstructureactiononpositivespin1}), (\ref{complexstructureactiononpositivespin2})).

The crucial observation is that, because $G_\tau$ does not act on the spinor part, it commutes with the operator $I_k^{S^+}$. Thus the above is
\[
\begin{split}
&\sum \langle   G_\tau I_k^{S^+} \Omega^t( \frac{\partial}{\partial \tau_a}, \frac{\partial}{\partial x_\mu}) c( dx_\mu)   \cdot{} \psi^i_\tau,     I_k^{S^+}\Omega^t(\frac{\partial}{\partial \tau_b}, \frac{\partial}{\partial x_\nu}) c(dx_\nu) \cdot{} \psi^j_\tau   \rangle d\tau_a \wedge  d\tau_b
\\
=& \sum \langle  I_k^{S^+} G_\tau  \Omega^t( \frac{\partial}{\partial \tau_a}, \frac{\partial}{\partial x_\mu}) c( dx_\mu)   \cdot{} \psi^i_\tau,     I_k^{S^+}\Omega^t(\frac{\partial}{\partial \tau_b}, \frac{\partial}{\partial x_\nu}) c(dx_\nu) \cdot{} \psi^j_\tau   \rangle d\tau_a \wedge  d\tau_b
\\
=&  \langle   G_\tau  \Omega^t \cdot{} \psi^i_\tau,     \Omega^t \cdot{} \psi^j_\tau   \rangle 
\end{split}
\]
as required. This verifies the ASD condition on $\langle   G_\tau  \Omega^t \cdot{} \psi^i_\tau,     \Omega^t \cdot{} \psi^j_\tau   \rangle $.
\end{proof}

\subsection{The inverse Nahm transform and the comparison map}\label{inverseNahmtransformcomparisonmap}

Let us assume that the family of HYM connections on $E'\simeq \mathcal{E}^\vee|_x\to X^\vee$ parametrised by $x\in X$ are all irreducible as in Section \ref{HyperkahlerperiodsontheMukaidualK3}.
Let $(\hat{ \mathcal{F} }, \hat{\alpha})$ be the Nahm transform of $(\mathcal{F}, \alpha)$, which has the same slope as $E'\to X^\vee$. Assume  $\text{H}^0(X^\vee, \underline{\Hom}(\mathcal{E}^\vee|_x, \mathcal{F} ))=0$ for all $x$, then we can perform the \textbf{inverse Nahm transform construction} starting from $\hat{\mathcal{F} }$. We consider the coupled Dirac operators
\[
\begin{split}
D^+_{\hat{\alpha}_x} : \Gamma(X^\vee, S^+_{X^\vee } \otimes \Hom(\mathcal{E}^\vee|_x, \hat{\mathcal{F}   }    ))\to \Gamma(X^\vee, S^-_{X^\vee } \otimes \Hom(\mathcal{E}^\vee|_x, \hat{\mathcal{F}  }     )), \\
D^-_{\hat{\alpha}_x} : \Gamma(X^\vee, S^-_{X^\vee } \otimes \Hom(\mathcal{E}^\vee|_x, \hat{\mathcal{F}    }   ))\to \Gamma(X^\vee, S^+_{X^\vee } \otimes \Hom(\mathcal{E}^\vee|_x, \hat{\mathcal{F} }      )).
\end{split}
\]
By Lemma \ref{vanishingkernellemma}, The kernel of $D^+_{\hat{\alpha}_x  }$ vanishes, and the kernel of $D^-_{\hat{\alpha}_x}$ fits into a vector bundle $\hat{\hat{\mathcal{F}}}\to X$, equipped with a natural connection $\hat{\hat{\alpha}}$. The pair $(\hat{\hat{\mathcal{F}}}, \hat{\hat{\alpha}})
$ is called  the inverse Nahm transform of $(\hat{F}, \hat{\alpha})$. The aim of this Section is to describe a \textbf{canonical comparison map} between the two bundles $\mathcal{F}\to X$ and $\hat{\hat{\mathcal{F}}}\to X$, by adapting ideas in \cite{BraamBaal}  where a similar construction is made over $T^4$. The main complication in our context is that the K3 metric is not flat. Still, the positive spinor bundle on $X$ and $X^\vee$ are flat, and their spaces of covariantly constant spinors are canonically identified.

Given $x\in X$ and $f\in F|_x^*$, 
we will construct out of $f$ a canonical section 
\[
G\Psi(f)\in \Gamma(X^\vee, \hat{\mathcal{F }}^*\otimes S^-_{X^\vee}\otimes \mathcal{E}^\vee|_x).
\] 
At any $\tau\in X^\vee$, let $s\in \hat{\mathcal{F} }|_\tau \subset \Gamma(X, \mathcal{F}\otimes \mathcal{E}^\vee|_\tau\otimes S^-_X)$. We can make $\Omega$ act on $s$ as in Section \ref{Nahmtransform}, to achieve a section
\[
\Omega^t\cdot s\in \Gamma(X, \mathcal{E}^\vee|_\tau \otimes \mathcal{F} \otimes S^+_X\otimes T^*_\tau X^\vee)=  \Gamma(X, \mathcal{E}^\vee|_\tau \otimes \mathcal{F} \otimes S^+_X)\otimes T^*_\tau X^\vee.
\]
We make the Green operator $G_\tau$ act on the $\Gamma(X, \mathcal{E}^\vee|_\tau \otimes \mathcal{F} \otimes S^+_X)$ factor of $\Omega^t\cdot s$ and leave the $T^*_\tau X^\vee$ factor untouched. We then contract $T_\tau^* X^\vee$ with $S^+_X$ by Clifford multiplication. This produces a section in 
$
\Gamma(X, \mathcal{E}^\vee|_\tau \otimes \mathcal{F} )\otimes S^-_{ X^\vee}|_\tau.
$
Now we evaluate the $\Gamma(X, \mathcal{E}^\vee|_\tau \otimes \mathcal{F} )$ factor at the point $x\in X$ against $f\in F_x^*$. The result lands in $ \mathcal{E}^\vee|_{\tau,x}  \otimes S^-_{ X^\vee}|_\tau$. This algorithm defines an element of $\hat{\mathcal{F}}^*\otimes \mathcal{E}^\vee|_{\tau,x}  \otimes S^-_{ X^\vee}|_\tau$. When we vary $\tau$ we get a section of $
\Gamma(X^\vee, \hat{\mathcal{F }}^*\otimes S^-_{X^\vee}\otimes \mathcal{E}^\vee|_x)
$ as promised, depending only on $f$. We denote this section as $G\Psi(f)$.

The main result of this Section is
\begin{prop}\label{comparisonmapDiracequation}(Compare Proposition 2.1 in \cite{BraamBaal})
The canonical section $G\Psi(f)$ satisfies the Dirac equation.
\end{prop}

This allows us to build the comparison map $\mathcal{F}\to \hat{\hat{\mathcal{F} }}$. Recall from the proof of Theorem \ref{NahmtransformperservesASD} the local orthonormal frame $\hat f^j= \psi_\tau^j$ of $\hat {\mathcal{F} }$. The section $G\Psi (f)$ evaluates on $\psi_\tau^j $ and outputs a local section of $S^-_{X^\vee}\otimes \mathcal{E}^\vee|_x$. There is an antilinear bundle automorphism $\epsilon$ acting on $S^-_{X^\vee}$ (\cf Appendix), which canonically extends to an antilinear bundle map $\epsilon: S^-_{X^\vee}\otimes \mathcal{E}^\vee|_x \to S^-_{X^\vee}\otimes \mathcal{E}|_x$ by using the Hermitian structure on $\mathcal{E}|_x$. Thus we can write a well defined section of $\hat{\mathcal{F} }\otimes S^-_{X^\vee}\otimes \mathcal{E}|_x\to X^\vee$, whose value at $\tau$ is
\begin{equation}\label{comparisonmap}
u_x(f) (\tau)= \frac{1}{2}\sum_j \epsilon(  G\Psi(f)|_\tau (\psi^j_\tau   )     ) \hat{f}^j.
\end{equation}
This section $u_x(f)$ depends on $f$ in an antilinear fashion. Therefore using the inner product on $\mathcal{F}|_x$, this construction gives a complex linear map from $\mathcal{F}|_x$ to $\Gamma(X^\vee,\hat{\mathcal{F} }\otimes S^-_{X^\vee}\otimes \mathcal{E}|_x )$.

\begin{thm}\label{comparisonmapDiracequation1}
The section $u_x(f)$ defined above satisfies the Dirac equation, and therefore is an element of $\hat{\hat{\mathcal{F}}}$.
\end{thm}

\begin{proof}
This follows immediately from Proposition \ref{comparisonmapDiracequation} because $\epsilon$ preserves the Clifford multiplication and the Hermitian connection, so preserves the Dirac equation.
\end{proof}

The rest of this Section is devoted to the proof of Proposition \ref{comparisonmapDiracequation}. This requires us to understand how to take covariant derivatives of $G\Psi(f)$ as $\tau\in X^\vee$ varies. We follow the definition of $G\Psi(f)$. The most essential step is to compute $G_\tau \Omega^t\cdot s$. We can regard $s\mapsto G_\tau \Omega^t\cdot s$ as an element in the Hom bundle from $\hat{\mathcal{F}}$ to the infinite rank bundle with fibre $\Gamma(X, \mathcal{E}^\vee|_\tau \otimes \mathcal{F}\otimes S^+_X)\otimes T^*_\tau X^\vee$. The covariant derivative on this Hom bundle is induced from the covariant derivative on $\hat{\mathcal{F}}$ and the covariant derivative on the infinite rank bundle; the latter is trivial on the $\mathcal{F}\otimes S^+_X$ factor, agrees with the Levi-Civita connection on the $T^*X^\vee$ factor, and is induced from the transposed universal connection $-\nabla^{univ,t}$ on the $\mathcal{E}^\vee$ factor. In concrete formulae, we need to compute
\[
\nabla_{\frac{\partial}{\partial \tau_i}} \{  G_\tau (\Omega^t\cdot s) \}=[-\nabla^{univ,t}_{\frac{\partial}{\partial \tau_i} } , G_\tau     ](\Omega^t\cdot s  )+G_\tau( \nabla_{\frac{\partial}{\partial \tau_i}  } \Omega^t )\cdot s+ G_\tau ( \Omega^t\cdot (-\nabla^{univ,t}_{\frac{\partial}{\partial \tau_i}  }s)   )
\]
on the infinite rank bundle, and subtract off 
\[
\begin{split}
G_\tau (  \Omega^t\cdot\{ \nabla^{\hat{\mathcal{F} } }_{\frac{\partial}{\partial \tau_i}} s       \}              )&=G_\tau   \Omega^t\cdot\{ -\nabla^{univ,t }_{\frac{\partial}{\partial \tau_i}} s- D^+_{\alpha_\tau} G_\tau D^-_{\alpha_\tau}( -\nabla^{univ,t}_{\frac{\partial}{\partial \tau_i}} s  )      \}              
\\
&=G_\tau   \Omega^t\cdot\{ -\nabla^{univ,t }_{\frac{\partial}{\partial \tau_i}} s- D^+_{\alpha_\tau} G_\tau [D^-_{\alpha_\tau}, -\nabla^{univ,t}_{\frac{\partial}{\partial \tau_i}}]s         \}     \\
&=   G_\tau   \Omega^t\cdot\{ -\nabla^{univ,t }_{\frac{\partial}{\partial \tau_i}} s- D^+_{\alpha_\tau} G_\tau (\iota_{\frac{\partial}{\partial \tau_i}}\Omega^t)\cdot s         \}      
.
\end{split}
\]
The above equation uses the description of the connection on $\hat{\mathcal{F} }$, as discussed in Section \ref{Nahmtransform}. After this subtraction, we get
\begin{equation}\label{GPsiderivativeintermediatestep}
[-\nabla^{univ,t}_{\frac{\partial}{\partial \tau_i} } , G_\tau     ](\Omega^t\cdot s  )+ G_\tau   \Omega^t\cdot\{ D^+_{\alpha_\tau} G_\tau (\iota_{\frac{\partial}{\partial \tau_i}}\Omega^t)s         \}+G_\tau( \nabla_{\frac{\partial}{\partial \tau_i}  } \Omega^t )\cdot s .            
\end{equation}
The task is to understand the individual terms.

\begin{lem}(Compare Lemma 2.2 in \cite{BraamBaal}, `partial derivative of the Green operator')
The commutator \[[-\nabla^{univ,t}_{\frac{\partial}{\partial \tau_i} } , G_\tau     ]=2 G_\tau ( -\iota_{\frac{\partial}{\partial \tau_i} }\Omega^t , \nabla_{\alpha_\tau}  ) G_\tau,\] 
where $( -\iota_{\frac{\partial}{\partial \tau_i} }\Omega^t , \nabla_{\alpha_\tau}  )=\sum_{\mu}-\Omega^t( \frac{\partial}{\partial \tau_i}  , \frac{\partial}{\partial x_\mu}) \nabla^{\alpha_\tau}_ \frac{\partial}{\partial x_\mu} $ for an orthonormal basis $\frac{\partial}{\partial x_\mu}$, or in other words, we contract the 1-form part of $ -\iota_{\frac{\partial}{\partial \tau_i} }\Omega^t$ and $\nabla_{\alpha_\tau} $.
\end{lem}

\begin{proof}
We compute the variation of the Laplacian $[-\nabla^{univ,t}_{\frac{\partial}{\partial \tau_i} } , \nabla^*_{\alpha_\tau} \nabla_{\alpha_\tau}     ]$. By the Jacobi identity, 
\[
[  \nabla^*_{\alpha_\tau} \nabla_{\alpha_\tau} ,   -\nabla^{univ,t}_{\frac{\partial}{\partial \tau_i} }      ]= \nabla^*_{\alpha_\tau}[ \nabla_{\alpha_\tau}, -\nabla^{univ,t}_{\frac{\partial}{\partial \tau_i} }    ]+[ \nabla^*_{\alpha_\tau}, -\nabla^{univ,t}_{\frac{\partial}{\partial \tau_i} }  ]    \nabla_{\alpha_\tau}.
\]
The first term is $\nabla^*_{\alpha_\tau}\circ \iota_{\frac{\partial}{\partial \tau_i} } \Omega^t$, where $\iota_{\frac{\partial}{\partial \tau_i} } \Omega^t$ should be understood as a pointwise curvature type operator acting on sections in $\Gamma(X, \mathcal{F}\otimes \mathcal{E}^\vee|_\tau)$. Since $\iota_{\frac{\partial}{\partial \tau_i} } \Omega^t$ is in the Coulumb gauge (\cf the proof of Lemma \ref{infinitesimalvariationofASDconnectionisinCoulumbgauge} and switch the role of $X$ and $X^\vee$), we get the operator identity
\[
\nabla^*_{\alpha_\tau}[ \nabla_{\alpha_\tau}, -\nabla^{univ,t}_{\frac{\partial}{\partial \tau_i} }    ]=(  -\iota_{\frac{\partial}{\partial \tau_i} }\Omega^t , \nabla_{\alpha_\tau}       ).
\]
Now for the second term in the Jacobi identity, we compute
\[
[ \nabla^*_{\alpha_\tau}, -\nabla^{univ,t}_{\frac{\partial}{\partial \tau_i} }  ] =\sum_{\mu} \iota_{\frac{\partial}{\partial x_\mu}   }  \circ \Omega^t( \frac{\partial}{\partial x_\mu},\frac{\partial}{\partial \tau_i}    ),
\] 
so the second term $[ \nabla^*_{\alpha_\tau}, -\nabla^{univ,t}_{\frac{\partial}{\partial \tau_i} }  ]    \nabla_{\alpha_\tau}=(  -\iota_{\frac{\partial}{\partial \tau_i} }\Omega^t , \nabla_{\alpha_\tau}       ).$ Thus
\[
[  \nabla^*_{\alpha_\tau} \nabla_{\alpha_\tau} ,   -\nabla^{univ,t}_{\frac{\partial}{\partial \tau_i} }      ]=2(  -\iota_{\frac{\partial}{\partial \tau_i} }\Omega^t , \nabla_{\alpha_\tau}           ), 
\]
and multiplying $G_\tau$ both on the left and the right gives us the Lemma. The reader can understand this as essentially the formula for the derivative of the inverse matrix.
\end{proof}

%\begin{rmk}
%The reason we refer to the commutator as the partial derivative, is that we can think of the universal connection as giving an infinitesimal product structure of $\mathcal{E}\to X\times X^\vee$ around the central fibre $\mathcal{E}|_\tau \to X$, and then the commutator is just $\frac{\partial G_\tau}{\partial \tau_i}$. This intuition is quite helpful as long as one remembers the `mixed partial derivatives' are not equal.  We notice that this operator $\frac{\partial G_\tau}{\partial \tau_i}$  does not interfere with spinors over $X$.
%\end{rmk}

\begin{lem}\label{computationOmegaDirac}
The operator
\[
\iota_{\frac{\partial}{\partial \tau_i} }\Omega^t\cdot D^+_{\alpha_\tau}=-(  \iota_{\frac{\partial}{\partial \tau_i} }\Omega^t , \nabla_{\alpha_\tau}           )+\sum_{k=1}^3 I_k^{S^+} (  \iota_{I_k\frac{\partial}{\partial \tau_i} }\Omega^t , \nabla_{\alpha_\tau}           ).
\]
\end{lem}

\begin{proof}
We can write using a pointwise orthonormal frame $\frac{\partial}{\partial x_\mu}$ that
\[
\iota_{\frac{\partial}{\partial \tau_i} }\Omega^t\cdot= \sum_{\mu} \Omega^t(\frac{\partial}{\partial \tau_i}, \frac{\partial}{\partial x_\mu}    )c(\frac{\partial}{\partial x_\mu}  ),
\quad
D^+_{\alpha_\tau} =\sum_{\nu}c(\frac{\partial}{\partial x_\nu}     )\nabla^{\alpha_\tau}_{\frac{\partial}{\partial x_\nu}  } .
\]
Now we write the composed operator $\iota_{\frac{\partial}{\partial \tau_i} }\Omega^t\cdot D^+_{\alpha_\tau} $ as the sum of $4$ parts, corresponding to $\frac{\partial}{\partial x_\mu}=\frac{\partial}{\partial x_\nu}$, and $\frac{\partial}{\partial x_\mu}=I_k\frac{\partial}{\partial x_\nu}$, for $k=1,2,3$. This leads to our claimed formula using the triholomorphic property of $\Omega^t$ and the properties of $I_k^{S^+}$. The manipulations are similar to the proof of Theorem \ref{NahmtransformperservesASD}. We leave the details to the reader as an exercise.
\end{proof}

The above two lemmas give
\begin{cor}\label{computationGOmegaDGOmega}
The term 
$
G_\tau   \Omega^t\cdot\{ D^+_{\alpha_\tau} G_\tau (\iota_{\frac{\partial}{\partial \tau_i}}\Omega^t)s         \}
$ in (\ref{GPsiderivativeintermediatestep}) is equal to
\[
\frac{1}{2}\sum_{j} d\tau_j \otimes \{  [ -\nabla^{univ,t}_{\frac{\partial}{\partial \tau_j} } , G_\tau       ] (\iota_{\frac{\partial}{\partial \tau_i}}\Omega^t)s +\sum_{k=1}^3   [ \nabla^{univ,t}_{I_k\frac{\partial}{\partial \tau_j} } , G_\tau       ] (\iota_{I_k\frac{\partial}{\partial \tau_i}}\Omega^t)s         \}.
\]
\end{cor}

To proceed further, we follow the definition of $G\Psi(f)$ to perform the contraction $T^*_\tau X^\vee \otimes S^+_{X^\vee} \to S^-_{X^\vee}$, which by the Leibniz rule is compatible with  taking  covariant derivative on $X^\vee$. Then we need to evaluate against the covector $f\in F^*_x$ to get the covariant derivative of $G\Psi(f)$; this is a trivial step, so we will somtimes suppress that to save some writing. To verify the Dirac equation, we need to Clifford multiply the contraction of (\ref{GPsiderivativeintermediatestep}) by $c(d\tau_i)$, and sum up $i=1,2,3,4$ to evaluate the Dirac operator on $G\Psi(f)$, and evantually show that everything cancels out to give us zero. Let us now analyse how this works for the $3$ terms in (\ref{GPsiderivativeintermediatestep}).

The first term gives $\sum_{i} c(d\tau_i) [- \nabla^{univ,t}_{\frac{\partial}{\partial \tau_i} } , G_\tau   ] (\Omega^t\cdot s)$. The second term gives
\[
\begin{split}
&\sum_i c(d\tau_i) G_\tau   \Omega^t\cdot\{ D^+_{\alpha_\tau} G_\tau (\iota_{\frac{\partial}{\partial \tau_i}}\Omega^t)\cdot s  )       \} \\
=&
\frac{1}{2}\sum_{i, j}c(d\tau_i) c(d\tau_j)  \{  [- \nabla^{univ,t}_{\frac{\partial}{\partial \tau_j} } , G_\tau       ] (\iota_{\frac{\partial}{\partial \tau_i}}\Omega^t)s +\sum_{k=1}^3   [ \nabla^{univ,t}_{I_k\frac{\partial}{\partial \tau_j} } , G_\tau       ] (\iota_{I_k\frac{\partial}{\partial \tau_i}}\Omega^t)s         \}\\
=&
\frac{1}{2}\sum_{i, j} \{c(d\tau_i) c(d\tau_j) -\sum_{k=1}^3 c(I_k d\tau_i) c(I_k d\tau_j)
\}
 \{  [- \nabla^{univ,t}_{\frac{\partial}{\partial \tau_j} } , G_\tau       ] (\iota_{\frac{\partial}{\partial \tau_i}}\Omega^t)s \} \\
=&
-\sum_{i, j} c(d\tau_j) c(d\tau_i) 
\{  [- \nabla^{univ,t}_{\frac{\partial}{\partial \tau_j} } , G_\tau       ] (\iota_{\frac{\partial}{\partial \tau_i}}\Omega^t)s \} \\
=& \sum_{j} c(d\tau_j)
  [ \nabla^{univ,t}_{\frac{\partial}{\partial \tau_j} } , G_\tau       ] \Omega^t\cdot s.
\end{split}
\]
Thus the second term exactly cancels with the first term. The third term gives
\begin{equation}\label{comparisonmapDiracequationthirdterm}
\sum_{i} c(d\tau_i)G_\tau( \nabla_{\frac{\partial}{\partial \tau_i}  } \Omega^t )\cdot s= \sum_{i} G_\tau c(d\tau_i)( \nabla_{\frac{\partial}{\partial \tau_i}  } \Omega^t )\cdot s,
\end{equation}
since $G_\tau$ does not interfere with spinors on $X^\vee$. To clarify, the covariant derivative on $\Omega^t$ is defined by combining the universal connection on $ad(\mathcal{E})$, with the Levi-Civita connection for 1-forms on $X^\vee$. 
We compute by the relation between Clifford multiplication and wedge product
\[
\sum_i c(d\tau_i)( \nabla_{\frac{\partial}{\partial \tau_i}  } \Omega^t )\cdot s= c( d\tau_i \wedge \nabla_{\frac{\partial}{\partial \tau_i}  } \Omega^t    )\cdot s + ( d_{A^{univ}}^ { X^\vee, *}\Omega^t  )\cdot s= c( d_{A^{univ}}^{X^\vee}\Omega^t )\cdot s ,
\]
where we used the component form of the Yang-Mills equation (\ref{YangMillsoncomponents}) to conclude $ d_{A^{univ}}^ { X^\vee, *}\Omega^t =0$. To understand $d_{A^{univ}}^ { X^\vee}\Omega^t $, we start from the Bianchi identity
\[
d_{A^{univ}} F(\nabla^{univ})=0,
\]
and decompose it into $X, X^\vee$ types, to see
\[
d_{A^{univ}}^ { X^\vee}\Omega=-d_{A^{univ}}^ { X} F(\nabla^{univ})^{X^\vee, X^\vee}= -\sum_j dx_j\nabla^{univ}_{  \frac{\partial}{\partial x_j}       } F(\nabla^{univ})^{X^\vee, X^\vee}.
\]
Now since $ F(\nabla^{univ})^{X^\vee, X^\vee}$ is projectively ASD on $X^\vee$ and its trace is independent of $x$, its variation \[\nabla^{univ}_{  \frac{\partial}{\partial x_j}       } F(\nabla^{univ})^{X^\vee, X^\vee}\] must be a coupled ASD 2-form, so the action on the positive spinor $dx_j\cdot s$ on $X^\vee$ is zero. This discussion shows the contribution from the third term (\ref{comparisonmapDiracequationthirdterm}) is zero. The upshot is that we have verified the Dirac equation in Proposition \ref{comparisonmapDiracequation}.

\subsection{Injective isometry}\label{Injectiveisometry}

The aim of this Section is to show the canonical map $u: f\mapsto u_x(f)$ (\cf equation (\ref{comparisonmap})) is an injective isometry, by adapting calculations in \cite{BraamBaal}. This requires us to first unravel the rather difficult definition of the Hermitian norm on $\hat{\hat{\mathcal{F}}}|_x$, which is inherited from $\hat{\hat{\mathcal{F}}}|_x\subset \Gamma(X^\vee, \hat{\mathcal{F}}\otimes S^-_{X^\vee}\otimes \mathcal{E}|_x)$. By definition,
\[
4\langle u_x(f), u_x(f')\rangle= \int_{X^\vee} \sum_j \langle \epsilon(G\Psi(f)|_\tau(\psi^j_\tau) ), \epsilon(G\Psi(f')|_\tau(\psi^j_\tau)) \rangle_{ S^-_{X^\vee}|_\tau\otimes \mathcal{E}|_{x,\tau } } d\text{Vol}_{X^\vee}.
\]
Now we will try to understand the integrand in the above expression.

\begin{lem}\label{HermitianisometryLemma1}
The integrand 
\[
\sum_j \langle \epsilon(G\Psi(f)|_\tau(\psi^j_\tau) ), \epsilon(G\Psi(f')|_\tau(\psi^j_\tau)) \rangle_{ S^-_{X^\vee}|_\tau\otimes \mathcal{E}|_{x,\tau } }
\]	
is equal to 
\[
\Tr_{ \mathcal{E}^\vee|_{x,\tau }       }
\langle f'  ,   f\circ \Tr_{ S^-_{X^\vee}|_\tau} G_\tau \Omega^t\cdot P_\tau (\Omega^t\cdot)^\dag G_\tau\rangle.
\]
Here $\langle f'  ,   f\circ \Tr_{ S^-_{X^\vee}|_\tau} G_\tau \Omega^t\cdot P_\tau (\Omega^t\cdot)^\dag G_\tau\rangle$ means evaluating the bundle-valued functional $ f\circ \Tr_{ S^-_{X^\vee}|_\tau} G_\tau \Omega^t\cdot P_\tau (\Omega^t\cdot)^\dag G_\tau$ against $f'\in \mathcal{F}^*|_x\simeq \mathcal{F}|_x$ to obtain a matrix in $\End( \mathcal{E}^\vee|_{x,\tau }  )$.
\end{lem}

\begin{proof}
Since $\epsilon$ is an antilinear isometry, we can rewrite the integrand as
\[
\sum_j \langle G\Psi(f')|_\tau(\psi^j_\tau) , G\Psi(f)|_\tau(\psi^j_\tau) \rangle_{ S^-_{X^\vee}|_\tau\otimes \mathcal{E}^\vee|_{x,\tau } }.
\]
This expression involves summing over an orthonormal basis in the kernel of the Dirac equation. We can rewrite this as an infinite sum over an orthonormal basis $\psi'^{j}_\tau$ of $\Gamma(X, \Hom(\mathcal{E}|_\tau, \mathcal{F}\otimes S^-_X))$, by inserting the projection operator $P_\tau=1-D_{\alpha_\tau}^+ G_\tau D^-_{\alpha_\tau}$. The result is
\[
\sum_{j=1}^\infty \langle f'\circ G_\tau \Omega^t\cdot \psi'^j_\tau, f\circ G_\tau \Omega^t\cdot P_\tau \psi'^j_\tau\rangle_{  S^-_{X^\vee}|_\tau\otimes \mathcal{E}^\vee|_{x,\tau }       }.
\]
We interpret $f\circ G_\tau \Omega^t\cdot P_\tau$ as a $S^-_{X^\vee}|_\tau\otimes \mathcal{E}^\vee|_{x,\tau } $ valued functional on the function space $L^2(X, \Hom(\mathcal{E}|_\tau, \mathcal{F}\otimes S^-_X)   )$. Thus the above expression is
\begin{equation}\label{correlatorfunction1}
\langle f'\circ G_\tau \Omega^t\cdot  , f\circ G_\tau \Omega^t\cdot P_\tau \rangle_{  S^-_{X^\vee}|_\tau\otimes \mathcal{E}^\vee|_{x,\tau }  \otimes L^2(X, \Hom(\mathcal{E}|_\tau, \mathcal{F}\otimes S^-_X|_\tau)  ) ^*    }.
\end{equation}
where we used the inner product on the finite dimensional vector space $S^-_{X^\vee}|_\tau\otimes \mathcal{E}^\vee|_{x,\tau }$ and the inner product on the linear functionals.
To simplify this further, we need to calculate the adjoint of the operator 
\[
\Omega^t \cdot: L^2(X, \Hom(\mathcal{E}|_\tau, \mathcal{F}\otimes S^-_X)  )\to L^2(X, \Hom(\mathcal{E}|_\tau, \mathcal{F} ) \otimes S^-_{X^\vee}|_\tau ),
\]
 \[\Omega^t\cdot= \sum_{i,j} \Omega^t( \frac{\partial }{\partial \tau_j}, \frac{\partial }{\partial x_i}) c(d\tau_j) c(dx_i), 
 \]
which is 
\[
(\Omega^t \cdot)^\dagger:  L^2(X, \Hom(\mathcal{E}|_\tau, \mathcal{F} ) \otimes S^-_{X^\vee}|_\tau )\to L^2(X, \Hom(\mathcal{E}|_\tau, \mathcal{F}\otimes S^-_X)  )
\]
\[
(\Omega^t\cdot)^\dag =\sum_{i,j} \Omega^t(\frac{\partial }{\partial x_i} , \frac{\partial }{\partial \tau_j})  c(dx_i)c(d\tau_j).
\]
This allows us to rewrite (\ref{correlatorfunction1}) as
\[
\langle f'\circ G_\tau   , f\circ G_\tau \Omega^t\cdot P_\tau (\Omega^t\cdot)^\dag\rangle_{  S^-_{X^\vee}|_\tau\otimes \mathcal{E}^\vee|_{x,\tau }  \otimes L^2(X, \Hom(\mathcal{E}|_\tau, \mathcal{F})\otimes S^-_{X^\vee}|_\tau)   ^*    }.
\]
Using that $G_\tau$ is self-adjoint, the above is
\[
\langle f'  , f\circ G_\tau \Omega^t\cdot P_\tau (\Omega^t\cdot)^\dag G_\tau\rangle_{  S^-_{X^\vee}|_\tau\otimes \mathcal{E}^\vee|_{x,\tau }  \otimes L^2(X, \Hom(\mathcal{E}|_\tau, \mathcal{F})\otimes S^-_{X^\vee}|_\tau)   ^*    }.
\]
As an analytic subtle point, the evaluation functional $f'$ is not bounded on the $L^2$ space, but we can still make sense of the above expression, because the presence of $P_\tau$ implies the functional $f\circ G_\tau \Omega^t\cdot P_\tau (\Omega^t\cdot)^\dag G_\tau$ is represented by a smooth bundle-valued function, and we just need to evaluate this function against the vector $f'\in \mathcal{F}^*|_x \simeq \mathcal{F}|_x$, and then contract the $S^-_{X^\vee}|_\tau\otimes \mathcal{E}^\vee|_{x,\tau }$ factor.

Finally, we observe that $f, f'$ do not interfere with the spinor factor. This means we can first calculate the spinor trace on $S^-_{X^\vee}|_\tau$, and then make the above evaluation, and contract the $ \mathcal{E}^\vee|_{x,\tau }$ factor. The result is
\[
\Tr_{ \mathcal{E}^\vee|_{x,\tau }       }
\langle f'  ,   f\circ \Tr_{ S^-_{X^\vee}|_\tau} G_\tau \Omega^t\cdot P_\tau (\Omega^t\cdot)^\dag G_\tau\rangle
\]
as required.
\end{proof}

We next deal with the expression $G_\tau \Omega^t\cdot P_\tau (\Omega^t\cdot)^\dag G_\tau$. The following Lemma is the analogue of Lemma 2.6 in \cite{BraamBaal}, although the non-flat nature of the K3 metric has made the calculations significantly more difficult.

\begin{lem}\label{HermitianisometryLemma2}
In a geodesic coordinate $\tau_i$ on $X^\vee$, we have
\begin{equation}
\Tr_{S^-_{X^\vee}|_\tau} G_\tau \Omega^t\cdot P_\tau (\Omega^t)^\dag G_\tau=-4\sum_i [ \nabla^{univ,t}_{\frac{\partial}{\partial \tau_i}  }, [ \nabla^{univ,t}_{\frac{\partial}{\partial \tau_i}  }  , G_\tau]        ]	.
\end{equation}
\end{lem}

\begin{proof}
We use $P_\tau=1-D_{\alpha_\tau}^+ G_\tau D^-_{\alpha_\tau}$ to write
\[
G_\tau \Omega^t\cdot P_\tau (\Omega^t\cdot)^\dag G_\tau= G_\tau \Omega^t\cdot  (\Omega^t\cdot)^\dag G_\tau- G_\tau \Omega^t \cdot D_{\alpha_\tau}^+ G_\tau D^-_{\alpha_\tau}(\Omega^t\cdot)^\dag G_\tau.
\]
By Lemma \ref{computationOmegaDirac}, the term $ G_\tau \Omega^t \cdot D_{\alpha_\tau}^+ G_\tau D^-_{\alpha_\tau}(\Omega^t\cdot)^\dag G_\tau$ is equal to
\begin{equation}\label{computationGOmegaDGDOmegaG1}
\begin{split}
&\sum_j c(d\tau_j) G_\tau  \{
-(\iota_{\frac{\partial}{\partial \tau_j} } \Omega^t, \nabla_{\alpha_\tau})+\sum_{k=1}^3 I_k^{S^+}( \iota_{I_k\frac{\partial}{\partial \tau_j} } \Omega^t, \nabla_{\alpha_\tau} )
\}G_\tau D^-_{\alpha_\tau}(\Omega^t\cdot)^\dag G_\tau \\
=& \sum_j 
G_\tau(-\iota_{\frac{\partial}{\partial \tau_j} } \Omega^t, \nabla_{\alpha_\tau})\{  c(d\tau_j)+\sum_{k=1}^3 c(I_kd\tau_j)I_k^{S^+}  \}G_\tau D^-_{\alpha_\tau}(\Omega^t\cdot)^\dag G_\tau \\
=& \frac{1}{2}\sum_j 
[  -\nabla^{univ,t}_{\frac{\partial}{\partial \tau_j}  }   , G_\tau       ]
\{  c(d\tau_j)+\sum_{k=1}^3 c(I_kd\tau_j)I_k^{S^+}  \} D^-_{\alpha_\tau}(\Omega^t\cdot)^\dag G_\tau  \\
=& 2\sum_j 
[  -\nabla^{univ,t}_{\frac{\partial}{\partial \tau_j}  }   , G_\tau       ] c(d\tau_j)
 D^-_{\alpha_\tau}(\Omega^t\cdot)^\dag G_\tau
\end{split} 
\end{equation}
To simplify further, we need to calculate
\[
D^-_{\alpha_\tau}(\Omega^t\cdot)^\dag=( \Omega^t\cdot D^+_{\alpha_\tau} )^\dagger
=-\{  -(\iota_{\frac{\partial}{\partial \tau_j} } \Omega^t, \nabla_{\alpha_\tau})+\sum_{k=1}^3 I_k^{S^+}( \iota_{I_k\frac{\partial}{\partial \tau_j} } \Omega^t, \nabla_{\alpha_\tau} )         \}^\dag c(d\tau_j).
\]
A short calculation, using that $\iota_{\frac{\partial}{\partial \tau_j} } \Omega^t$ is in the Coulumb gauge (\cf Lemma \ref{infinitesimalvariationofASDconnectionisinCoulumbgauge}), shows that  $(\iota_{\frac{\partial}{\partial \tau_j} } \Omega^t, \nabla_{\alpha_\tau})$ is self-adjoint. We also have that $I_k^{S^+}$ is anti-self-adjoint. Thus
\[
\begin{split}
D^-_{\alpha_\tau}(\Omega^t\cdot)^\dag
=& \sum_j
\{ (\iota_{\frac{\partial}{\partial \tau_j} } \Omega^t, \nabla_{\alpha_\tau})-  \sum_{k=1}^3 I_k^{S^+}( \iota_{I_k\frac{\partial}{\partial \tau_j} } \Omega^t, \nabla_{\alpha_\tau} )                 \} c(d\tau_j) \\
=& \sum_i (\iota_{\frac{\partial}{\partial \tau_i} } \Omega^t, \nabla_{\alpha_\tau})\{  
 c(d\tau_i)+\sum_{l=1}^3 I_l^{S^+} c(I_l d\tau_i)\} \\
=& 4\sum_i (\iota_{\frac{\partial}{\partial \tau_i} } \Omega^t, \nabla_{\alpha_\tau})  
 c(d\tau_i).
\end{split}
\]
Substituting this into the above expression (\ref{computationGOmegaDGDOmegaG1}), we get
\begin{equation}\label{computationGOmegaFGOmegaG2}
\begin{split}
&\Tr_{S^-_{X^\vee}|_\tau} G_\tau \Omega^t \cdot D_{\alpha_\tau}^+ G_\tau D^-_{\alpha_\tau}(\Omega^t\cdot)^\dag G_\tau         \\
=& 8\sum_{i,j} 
[  -\nabla^{univ,t}_{\frac{\partial}{\partial \tau_j}  }   , G_\tau       ](\iota_{\frac{\partial}{\partial \tau_i} } \Omega^t, \nabla_{\alpha_\tau})G_\tau
 \{ \Tr_{S^-_{X^\vee}|_\tau} c(d\tau_j)   
 c(d\tau_i) \} \\
=& 
-16\sum_{i} 
[  -\nabla^{univ,t}_{\frac{\partial}{\partial \tau_i}  }   , G_\tau       ](\iota_{\frac{\partial}{\partial \tau_i} } \Omega^t, \nabla_{\alpha_\tau})G_\tau.
\end{split}
\end{equation}
We calculate further by the definition of the commutator, and by using Lemma \ref{computationOmegaDirac},
\[
\begin{split}
&[  -\nabla^{univ,t}_{\frac{\partial}{\partial \tau_i}  }   , G_\tau       ](\iota_{\frac{\partial}{\partial \tau_i} } \Omega^t, \nabla_{\alpha_\tau})G_\tau \\
=&- \nabla^{univ,t}_{\frac{\partial}{\partial \tau_i}  } G_\tau (\iota_{\frac{\partial}{\partial \tau_i} } \Omega^t, \nabla_{\alpha_\tau})G_\tau
+G_\tau \nabla^{univ,t}_{\frac{\partial}{\partial \tau_i}  }\circ (\iota_{\frac{\partial}{\partial \tau_i} } \Omega^t, \nabla_{\alpha_\tau})G_\tau \\
=& -\frac{1}{2}\nabla^{univ,t}_{\frac{\partial}{\partial \tau_i}  }[ \nabla^{univ,t}_{\frac{\partial}{\partial \tau_i}  }  , G_\tau] 
+G_\tau \nabla^{univ,t}_{\frac{\partial}{\partial \tau_i}  }\circ (\iota_{\frac{\partial}{\partial \tau_i} } \Omega^t, \nabla_{\alpha_\tau})G_\tau \\
=&-\frac{1}{2}\nabla^{univ,t}_{\frac{\partial}{\partial \tau_i}  }[ \nabla^{univ,t}_{\frac{\partial}{\partial \tau_i}  }  , G_\tau] + G_\tau [ \nabla^{univ,t}_{\frac{\partial}{\partial \tau_i}  },  (\iota_{\frac{\partial}{\partial \tau_i} } \Omega^t, \nabla_{\alpha_\tau})    ]G_\tau 
\\
&+ G_\tau (\iota_{\frac{\partial}{\partial \tau_i} } \Omega^t, \nabla_{\alpha_\tau})\nabla^{univ,t}_{\frac{\partial}{\partial \tau_i}  }\circ G_\tau \\
=& -\frac{1}{2}\nabla^{univ,t}_{\frac{\partial}{\partial \tau_i}  }[ \nabla^{univ,t}_{\frac{\partial}{\partial \tau_i}  }  , G_\tau] + G_\tau [ \nabla^{univ,t}_{\frac{\partial}{\partial \tau_i}  },  (\iota_{\frac{\partial}{\partial \tau_i} } \Omega^t, \nabla_{\alpha_\tau})    ]G_\tau 
\\
&+ G_\tau (\iota_{\frac{\partial}{\partial \tau_i} } \Omega^t, \nabla_{\alpha_\tau})[\nabla^{univ,t}_{\frac{\partial}{\partial \tau_i}  }, G_\tau]+ \frac{1}{2}[\nabla^{univ,t}_{\frac{\partial}{\partial \tau_i}  }  , G_\tau      ]\nabla^{univ,t}_{\frac{\partial}{\partial \tau_i}  } \\
=& -\frac{1}{2}[\nabla^{univ,t}_{\frac{\partial}{\partial \tau_i}  }, [ \nabla^{univ,t}_{\frac{\partial}{\partial \tau_i}  }  , G_\tau]] + G_\tau [ \nabla^{univ,t}_{\frac{\partial}{\partial \tau_i}  },  (\iota_{\frac{\partial}{\partial \tau_i} } \Omega^t, \nabla_{\alpha_\tau})    ]G_\tau 
\\
&+ G_\tau (\iota_{\frac{\partial}{\partial \tau_i} } \Omega^t, \nabla_{\alpha_\tau})[\nabla^{univ,t}_{\frac{\partial}{\partial \tau_i}  }, G_\tau].
\end{split}
\]
Comparing this with
\[
\begin{split}
[-\nabla^{univ,t}_{\frac{\partial}{\partial \tau_i}  }   , G_\tau       ](\iota_{\frac{\partial}{\partial \tau_i} } \Omega^t, \nabla_{\alpha_\tau})G_\tau 
=&-2 G_\tau (\iota_{\frac{\partial}{\partial \tau_i} } \Omega^t, \nabla_{\alpha_\tau})G_\tau (\iota_{\frac{\partial}{\partial \tau_i} } \Omega^t, \nabla_{\alpha_\tau})G_\tau \\
=& G_\tau (\iota_{\frac{\partial}{\partial \tau_i} } \Omega^t, \nabla_{\alpha_\tau})[-\nabla^{univ,t}_{\frac{\partial}{\partial \tau_i}  }   , G_\tau       ],
\end{split} 
\]
we obtain
\[
\begin{split}
& 2[-\nabla^{univ,t}_{\frac{\partial}{\partial \tau_i}  }   , G_\tau       ](\iota_{\frac{\partial}{\partial \tau_i} } \Omega^t, \nabla_{\alpha_\tau})G_\tau \\
=&-\frac{1}{2}[\nabla^{univ,t}_{\frac{\partial}{\partial \tau_i}  }, [ \nabla^{univ,t}_{\frac{\partial}{\partial \tau_i}  }  , G_\tau]] + G_\tau [ \nabla^{univ,t}_{\frac{\partial}{\partial \tau_i}  },  (\iota_{\frac{\partial}{\partial \tau_i} } \Omega^t, \nabla_{\alpha_\tau})    ]G_\tau.
\end{split} 
\]
Substituting this into (\ref{computationGOmegaFGOmegaG2}), we get
\begin{equation}\label{computationGOmegaFGOmegaG3}
\begin{split}
&\Tr_{S^-_{X^\vee}|_\tau} G_\tau \Omega^t \cdot D_{\alpha_\tau}^+ G_\tau D^-_{\alpha_\tau}(\Omega^t\cdot)^\dag G_\tau         \\
=& \sum_i 4[\nabla^{univ,t}_{\frac{\partial}{\partial \tau_i}  }, [ \nabla^{univ,t}_{\frac{\partial}{\partial \tau_i}  }  , G_\tau]] -8 G_\tau [ \nabla^{univ,t}_{\frac{\partial}{\partial \tau_i}  },  (\iota_{\frac{\partial}{\partial \tau_i} } \Omega^t, \nabla_{\alpha_\tau})    ]G_\tau.
\end{split}
\end{equation}
We work in a geodesic coordinate $\tau_i$ on $X^\vee$. Then the Coulumb condition (\ref{YangMillsoncomponents}) reads
\[
\sum_i \nabla^{univ,t}_{\frac{\partial}{\partial \tau_i}  }\iota_{\frac{\partial}{\partial \tau_i} } \Omega^t=0,
\]
hence 
\[
\sum_i [ \nabla^{univ,t}_{\frac{\partial}{\partial \tau_i}  },  (\iota_{\frac{\partial}{\partial \tau_i} } \Omega^t, \nabla_{\alpha_\tau})         ]= \sum_{i,\mu}
\Omega^t(\frac{\partial}{\partial \tau_i}, \frac{\partial}{\partial x_\mu})\circ \Omega^t(\frac{\partial}{\partial \tau_i}, \frac{\partial}{\partial x_\mu}),
\]
where $\frac{\partial }{\partial x_\mu}$ form an orthonormal basis at the given point. Using this, we can rewrite the RHS of (\ref{computationGOmegaFGOmegaG3}) as
\[
4\sum_i [\nabla^{univ,t}_{\frac{\partial}{\partial \tau_i}  }, [ \nabla^{univ,t}_{\frac{\partial}{\partial \tau_i}  }  , G_\tau]] -8 G_\tau 
\sum_{i,\mu}
\Omega^t(\frac{\partial}{\partial \tau_i}, \frac{\partial}{\partial x_\mu})\circ \Omega^t(\frac{\partial}{\partial \tau_i}, \frac{\partial}{\partial x_\mu})
G_\tau.
\]
However, a short computation using the triholomorphic property of $\Omega$ gives
\[
\begin{split}
& \Tr_{S^-_{X^\vee}|_\tau} G_\tau \Omega^t\cdot (\Omega\cdot)^\dag G_\tau 
= -8G_\tau\sum_{i,\mu}  \Omega^t(\frac{\partial}{\partial \tau_i}, \frac{\partial}{\partial x_\mu})\circ \Omega^t(\frac{\partial}{\partial \tau_i}, \frac{\partial}{\partial x_\mu})G_\tau.
\end{split}
\]
Combining the above two equations, we obtain
\[
\Tr_{S^-_{X^\vee}|_\tau} G_\tau \Omega^t\cdot P_\tau (\Omega^t)^\dag G_\tau=-4\sum_i [ \nabla^{univ,t}_{\frac{\partial}{\partial \tau_i}  }, [ \nabla^{univ,t}_{\frac{\partial}{\partial \tau_i}  }  , G_\tau]        ]
\]
as required.
\end{proof}

\begin{cor}
The inner product $\langle u_x(f), u_x(f')\rangle $ is equal to
\begin{equation}\label{doubletransformHermitiannorm1}
-\int_{X^\vee} \Tr_{\mathcal{E}^\vee|_{x,\tau} } \langle f', f\circ \sum_i [\nabla^{univ,t}_{\frac{\partial}{\partial \tau_i}}, [ \nabla^{univ,t}_{\frac{\partial}{\partial \tau_i}}, G_\tau       ]       ] \rangle d\text{Vol}_{X^\vee}.
\end{equation}
\end{cor}

\begin{rmk}
The expression $-\sum_i [\nabla^{univ,t}_{\frac{\partial}{\partial \tau_i}}, [ \nabla^{univ,t}_{\frac{\partial}{\partial \tau_i}}, G_\tau       ]       ]$ can be heuristically understood as the Laplacian of $G_\tau$ in the $X^\vee$ direction. If we pretend the Schwartz kernel $G_\tau(y,x)$ of the Green operator $G_\tau$ is smooth, then (\ref{doubletransformHermitiannorm1}) would be zero. But this is false, because $f'$ and $f$ are located at the same point $x\in X$, and $G_\tau(y,x)$ blows up at $y=x$. The delicate cancellations which took place in the above calculation say that despite the singular nature of $G_\tau$, the expression $ -\sum_i [\nabla^{univ,t}_{\frac{\partial}{\partial \tau_i}}, [ \nabla^{univ,t}_{\frac{\partial}{\partial \tau_i}}, G_\tau       ]       ]$ is smooth at $y=x$. 
\end{rmk}

Now we proceed to evaluate (\ref{doubletransformHermitiannorm1}) by a careful consideration of the \textbf{asymptotes} as $y\to x$. For this, we first introduce a topological trivialisation of $\mathcal{E}\to X\times X^\vee$ over a local $T\subset X$, so that for all $y\in T$, the underlying smooth bundle of $\mathcal{E}|_y\to X^\vee$ become identified. In other words, we define the parallel transport operators $Q_\tau(x,y): \mathcal{E}^\vee|_{y,\tau}\to \mathcal{E}^\vee|_{x,\tau}$. In particular $Q_\tau(x,x)=1$. This $Q_\tau$ corresponds to a local flat connection $\nabla^{Q_\tau}$ on $\mathcal{E}^\vee|_\tau \to T\subset X$. We denote $\nabla^{univ}_{\frac{\partial}{\partial y_i}  }=\nabla^{Q_\tau}_{\frac{\partial}{\partial y_i}   }+A_i$, where $A_i$ can be thought of as the connection matrix of $\nabla^{univ}|_{\mathcal{E}|_\tau  }$, and we can demand $A_i(x)=0$ at the particular point $y=x\in T\subset X$, although we cannot make $A_i$ vanish globally because $\nabla^{univ}$ is far from being flat.

\begin{lem}\label{HermitianisometryLemma3}
The expression (\ref{doubletransformHermitiannorm1}) is equal to the limit
\begin{equation}\label{doubletransformHermitiannorm2}
\lim_{y\to x} \int_{X^\vee} \langle f', f\circ \Tr_{\mathcal{E}|_{\tau,x} } \{ ( \Lap_{X^\vee}^{univ}  Q_\tau(x,y) )G_\tau (y,x)   \}  \rangle d\text{Vol}_{X^\vee},
\end{equation}
where $\Lap_{X^\vee}^{univ}$ is $-\sum_i\nabla^{univ}_{\frac{\partial}{\partial \tau_i} } \nabla^{univ}_{\frac{\partial}{\partial \tau_i} }$ in geodesic local coordinates on $X^\vee$.
\end{lem}

\begin{proof}
Recall $G_\tau(y,x)$ is the Schwartz kernel of $G_\tau$, so the Schwartz kernel of $[\nabla^{univ,t}_{\frac{\partial}{\partial \tau_i}}, [ \nabla^{univ,t}_{\frac{\partial}{\partial \tau_i}}, G_\tau       ]       ]$ is $\nabla^{univ}_{\frac{\partial}{\partial \tau_i}}\nabla^{univ}_{\frac{\partial}{\partial \tau_i}}G_\tau (y,x)$, where we are differentiating the $\mathcal{E}^\vee|_{y,\tau}\otimes \mathcal{E}|_{x,\tau}\otimes \mathcal{F}|_y \otimes \mathcal{F}|_x^*$ valued function $G_\tau(y,x)$ with respect to the parameter $\tau_i$, and when the connection acts on the dual bundle the minus transpose is implicitly understood. This is continuous at $y=x$ even though $G_\tau(y,x)$ is not. To take the trace over $\mathcal{E}|_{x,\tau}$, we need to multiply by the parallel transport operator $Q_\tau(x,y)$. This allows us to write (\ref{doubletransformHermitiannorm1}) as the limit
\[
\lim_{y\to x} \int_{X^\vee} \langle f', f\circ \Tr_{\mathcal{E}|_{\tau,x} }\{Q_\tau(x,y)   \Lap_{X^\vee}^{univ}   G_\tau (y,x)   \}  \rangle d\text{Vol}_{X^\vee}.
\]
For any $y\neq x$, all expressions are smooth, and an application of Green's formula gives the claim.
\end{proof}

%$\Tr_{\mathcal{E}^\vee|_{x,\tau} } \langle f', f\circ [\nabla^{univ,t}_{\frac{\partial}{\partial \tau_i}}, [ \nabla^{univ,t}_{\frac{\partial}{\partial \tau_i}}, G_\tau       ]       ] \rangle d\text{Vol}_{X^\vee}$

\begin{lem}
The leading asymptote as $y\to x$ of $\Lap^{univ}_{X^\vee} Q_\tau(x,y)$ in the trivialisation defined by $Q_\tau$ is
\begin{equation}
\begin{split}
\Lap^{univ}_{X^\vee} Q_\tau(x,y)
& \sim
\sum_{j,k}\{\frac{1}{2}\Lap^{univ}_{X^\vee} (\nabla^{Q_\tau}_{ \frac{\partial}{\partial y_k} }A_j(y))|_{y=x} \\
&-\sum_i \Omega(\frac{\partial}{\partial y_j},  { \frac{\partial}{\partial \tau_i}  }  )
\Omega(\frac{\partial}{\partial y_k},  { \frac{\partial}{\partial \tau_i}  }  )|_{y=x}\} y_jy_k+ O( |x-y|^3 ).
\end{split}
\end{equation}
Here in the coordinates $y_i$ the origin corresponds to the point $x$.
\end{lem}

\begin{proof}
Since $Q(x,x)=1$, we see $\Lap^{univ}_{X^\vee} Q_\tau(x,x)=0$.

For the first order derivative, we compute
\[
\begin{split}
\nabla^{Q_\tau}_{ \frac{\partial}{\partial y_j}  }\nabla^{univ}_{ \frac{\partial}{\partial \tau_i}  }  Q_\tau(x,y)=[\nabla^{Q_\tau}_{ \frac{\partial}{\partial y_j}  }, \nabla^{univ}_{ \frac{\partial}{\partial \tau_i}  }    ]Q_\tau(x,y) \\
=[\nabla^{univ}_{ \frac{\partial}{\partial y_j}  }, \nabla^{univ}_{ \frac{\partial}{\partial \tau_i}  }    ]Q_\tau(x,y)-[A_j, \nabla^{univ}_{ \frac{\partial}{\partial \tau_i}  }] Q_\tau(x,y) \\
=\Omega(\frac{\partial}{\partial y_j},  { \frac{\partial}{\partial \tau_i}  }  )Q_\tau(x,y)+ \nabla^{univ}_{ \frac{\partial}{\partial \tau_i}  }A_j(y) Q_\tau(x,y),
\end{split}
\]
hence
\[
\begin{split}
&\nabla^{Q_\tau}_{ \frac{\partial}{\partial y_j} }\nabla^{univ}_{ \frac{\partial}{\partial \tau_i}  } \nabla^{univ}_{ \frac{\partial}{\partial \tau_i}  } Q_\tau(x,y)  \\
=& [ \nabla^{Q_\tau}_{ \frac{\partial}{\partial y_j} }, \nabla^{univ}_{ \frac{\partial}{\partial \tau_i}  }       ]\nabla^{univ}_{ \frac{\partial}{\partial \tau_i}  } Q_\tau(x,y) + \nabla^{univ}_{ \frac{\partial}{\partial \tau_i}  } 
\{ \{
\Omega(\frac{\partial}{\partial y_j},  { \frac{\partial}{\partial \tau_i}  }  )+ \nabla^{univ}_{ \frac{\partial}{\partial \tau_i}  }A_j(y)\} Q_\tau(x,y)
\} \\
=& \{
\Omega(\frac{\partial}{\partial y_j},  { \frac{\partial}{\partial \tau_i}  }  )+ \nabla^{univ}_{ \frac{\partial}{\partial \tau_i}  }A_j(y)\}\nabla^{univ}_{ \frac{\partial}{\partial \tau_i}  } Q_\tau(x,y) \\
& + \nabla^{univ}_{ \frac{\partial}{\partial \tau_i}  } 
\{ \{
\Omega(\frac{\partial}{\partial y_j},  { \frac{\partial}{\partial \tau_i}  }  )+ \nabla^{univ}_{ \frac{\partial}{\partial \tau_i}  }A_j(y)\} Q_\tau(x,y)
\} \\
=& 2 \{
\Omega(\frac{\partial}{\partial y_j},  { \frac{\partial}{\partial \tau_i}  }  )+ \nabla^{univ}_{ \frac{\partial}{\partial \tau_i}  }A_j(y)\}\nabla^{univ}_{ \frac{\partial}{\partial \tau_i}  } Q_\tau(x,y)  \\
& + \{\nabla^{univ}_{ \frac{\partial}{\partial \tau_i}  } \Omega(\frac{\partial}{\partial y_j},  { \frac{\partial}{\partial \tau_i}  }  )+  \nabla^{univ}_{ \frac{\partial}{\partial \tau_i}  }\nabla^{univ}_{ \frac{\partial}{\partial \tau_i}  }A_j(y)\} Q_\tau(x,y).
\end{split}
\]
Summing over $i$ and applying the Coulumb gauge condition
\[
\sum_i \nabla^{univ}_{ \frac{\partial}{\partial \tau_i}  } \Omega(\frac{\partial}{\partial y_j},  { \frac{\partial}{\partial \tau_i}  }  )=0,
\]
We obtain
\begin{equation}\label{QtauLaplacianasymptote1}
\begin{split}
\nabla^{Q_\tau}_{ \frac{\partial}{\partial y_j} }\Lap^{univ}_{X^\vee} Q_\tau(x,y) 
=& (\Lap^{univ}_{X^\vee}A_j(y) )Q_\tau(x,y) \\
&- 2 \sum_i \{
\Omega(\frac{\partial}{\partial y_j},  { \frac{\partial}{\partial \tau_i}  }  )
+ \nabla^{univ}_{ \frac{\partial}{\partial \tau_i}  }A_j(y)\}\nabla^{univ}_{ \frac{\partial}{\partial \tau_i}  } Q_\tau(x,y) ,
\end{split}
\end{equation}
which being evaluated at $y=x$, gives
\[
\{\nabla^{Q_\tau}_{ \frac{\partial}{\partial y_j} }\Lap^{univ}_{X^\vee} Q_\tau(x,y)\}|_{y=x}=0.
\]
Now we proceed to evaluate the second derivatives, at $y=x$. By differentiating (\ref{QtauLaplacianasymptote1}) and commuting the operators, we get
\[
\begin{split}
\nabla^{Q_\tau}_{ \frac{\partial}{\partial y_k} }\nabla^{Q_\tau}_{ \frac{\partial}{\partial y_j}  }\Lap^{univ}_{X^\vee} Q_\tau(x,y) |_{y=x}
&= 
\Lap^{univ}_{X^\vee} (\nabla^{Q_\tau}_{ \frac{\partial}{\partial y_k} }A_j(y))|_{y=x} \\
&- 2\sum_i\Omega(\frac{\partial}{\partial y_j},  { \frac{\partial}{\partial \tau_i}  }  )
\Omega(\frac{\partial}{\partial y_k},  { \frac{\partial}{\partial \tau_i}  }  )|_{y=x}.
\end{split}
\]
We have thus obtained the Taylor coefficients up to the second order for the expression $\Lap^{univ}_{X^\vee} Q_\tau(x,y)$ as $y\to x$.
\end{proof}

We also need to recall the short distance asymptote of the Green's function $G_\tau(y,x)$ as $y\to x$ in geodesic coordinates $y_i$ on $T\subset X$, as in \cite{BraamBaal}, proof of Proposition 2.7. Recall also the connection matrix of $\alpha_\tau$ vanishes at $y=x$ in our chosen trivialisation.

\begin{lem}
In the geodesic coordinates on $X$ and using our trivialisation of bundles,
the asymptote as $y\to x$ of the Green's function $G_\tau(y,x)$ is
\begin{equation}\label{asymptotesGreenfunction}
G_\tau(y,x)\sim \frac{1}{4\pi^2 |y-x|^2} (I +O(|y-x|^{2-\epsilon}) ).
\end{equation}
where $|y-x|^2=\sum_i y_i^2$ and $\epsilon$ is any small positive number.
\end{lem}

\begin{lem}\label{HermitianisometryLemma4}
The integral 
\[
\int_{X^\vee} \Tr_{\mathcal{E}|_{x,\tau} }\{ (\Lap^{univ}_{X^\vee} Q_\tau(x,y))G_\tau(y,x) \} d\text{Vol}_{X^\vee}
\]
converges to the identity matrix on $\mathcal{F}|_x$ as $y\to x$.
\end{lem}

\begin{proof}
Using the asymptotes of $\Lap^{univ}_{X^\vee} Q_\tau(x,y)$ and $G_\tau(y,x)$, we see that the limit is given by
\[
\begin{split}
\sum_{j,k} \frac{y_jy_k I_{\mathcal{F}|_x} }{4\pi^2|y-x|^2}  \int_{X^\vee} \Tr_{\mathcal{E}|_{x,\tau} } \{\frac{1}{2}\Lap^{univ}_{X^\vee} (\nabla^{Q_\tau}_{ \frac{\partial}{\partial y_k} }A_j(y))|_{y=x} \\
- \sum_i \Omega(\frac{\partial}{\partial y_j},  { \frac{\partial}{\partial \tau_i}  }  )
\Omega(\frac{\partial}{\partial y_k},  { \frac{\partial}{\partial \tau_i}  }  )|_{y=x}\}d\text{Vol}_{X^\vee}.
\end{split}
\]
The integral 
\[
\int_{X^\vee} \Tr_{\mathcal{E}|_{x,\tau} } \Lap^{univ}_{X^\vee} (\nabla^{Q_\tau}_{ \frac{\partial}{\partial y_k} }A_j(y))|_{y=x} d\text{Vol}_{X^\vee}=0,
\]
because it is the integral of a Laplacian. The other integral
\[
\frac{-1}{4\pi^2}\int_{X^\vee} \Tr_{\mathcal{E}|_{x,\tau} } \sum_i \Omega(\frac{\partial}{\partial x_j},  { \frac{\partial}{\partial \tau_i}  }  )
\Omega(\frac{\partial}{\partial x_k},  { \frac{\partial}{\partial \tau_i}  }  ) d\text{Vol}_{X^\vee}=\delta_{jk},
\]
because the LHS is the same as the metric $g^{\vee\vee}(\frac{\partial}{\partial x_j}, \frac{\partial}{\partial x_k}   )$ on $X$ which we treated in Section \ref{ASDconnectionsonMukaidual}, and there we showed that it agrees with the metric $g(\frac{\partial}{\partial x_j}, \frac{\partial}{\partial x_k}    )$.
Thus the whole expression of the limit is
\[
\sum_{j,k} \frac{y_jy_k I_{\mathcal{F}|_x} }{|y-x|^2}\delta_{jk}= I_{\mathcal{F}|_x}
\]
as required.
\end{proof}

\begin{thm}
The antilinear map $f\mapsto u_x(f)$ is an isometry $\mathcal{F}^*|_x\to \hat{\hat{\mathcal{F}}}|_x$.
\end{thm}

\begin{proof}
By the Lemmas in this Section, the inner product $\langle u_x(f), u_x(f')\rangle$ is equal to (\ref{doubletransformHermitiannorm2}), which by the above Lemma is equal to $\langle f', f \rangle$. Thus the antilinear map $f\mapsto u_x(f)$ is an injective isometry. 

To show it is also an isomorphism, the family Atiyah-Singer theorem says that the Chern character (or equivalently the Mukai vector) of $\hat{\hat{\mathcal{F}} }$ can be specified by the Fourier-Mukai transform on cohomology
\[
v( \hat{\mathcal{F} }   )= FM( v(\mathcal{F}  )    ), \quad v( \hat{\hat{\mathcal{F}} }  )=FM^\vee( v(\hat{\mathcal{F} }) ).
\]
Using that $FM^\vee$ is inverse to $FM$, we see that the two bundles $\hat{\hat{\mathcal{F}}}$ and $\mathcal{F}$ have the same Mukai vector, and in particular have the same rank. Thus the injective isometry must be an isomorphism.
\end{proof}

\subsection{Comparing the connections}\label{Comparingtheconnections}

The aim of this Section is to compare the connection $\hat{\hat{\alpha}}$ on the inverse Nahm transform $\hat{\hat{\mathcal{F}}}$, with the connection $\alpha$ on the original bundle $\mathcal{F}$.

Let $f$ be a local section of the dual bundle $\mathcal{F}^*$. Via the canonical comparison map, this gives us a section of $\hat{\hat{\mathcal{F}}}$, whose value at each $x\in X$ is just 
\[
u_x(f)= \frac{1}{2}\sum_j \epsilon(G\Psi(f)|_\tau(\psi^j_\tau) )\otimes \hat{f}^j \in \hat{\hat{\mathcal{F}}}|_x\subset \Gamma(X^\vee, \hat{\mathcal{F}}\otimes S^-_{X^\vee}\otimes \mathcal{E}|_x).
\]
We need to understand the meaning of covariant derivatives on the bundle $\hat{\hat{\mathcal{F}}}$. This is entirely analogous to the definition of $\hat{\alpha}$. Given any section $s$ of $\hat{\hat{\mathcal{F}}}\to X$, we first regard it as a section of the infinite rank bundle $\hat{\hat{\mathcal{H}}}$ with fibres
\[
\hat{\hat{\mathcal{H}}}|_x= \Gamma(X^\vee, \hat{\mathcal{F}}\otimes S^-_{X^\vee}\otimes \mathcal{E}|_x),
\]
take the natural covariant derivative of $s$ on the bundle $\hat{\hat{\mathcal{H}}}\to X$, and then project orthogonally to the subbundle $\hat{\hat{\mathcal{F}}}$. The connection matrix on $\hat{\hat{\mathcal{F}}}$ is therefore specified by $\langle u_x(f), \nabla_x u_x(f')\rangle$ for any local sections $f, f'$ of $\mathcal{F}$, which without loss of generality satisfy $\nabla^{\alpha}f=0$ and $\nabla^{\alpha}f'=0$ at the given point $x\in X$.

The following delicate computation follows the strategy of Section \ref{Injectiveisometry} of evaluating some integrals by asymptotic expansion.

\begin{prop}\label{comparingconnectionmatrices}
At $x\in X$, if $f$ and $f'$ have vanishing covariant derivatives, then for $\mu=1,2,3,4$,
\[\langle u_x (f), \nabla_{\frac{\partial}{\partial x_\mu}   }u_x(f')\rangle=0. \]
In other words, the connection matrices of $\alpha$ and $\hat{\hat{\alpha}}$ agree under the canonical comparison map.
\end{prop}

\begin{proof}
By the definition of the canonical comparison map
\begin{equation}\label{inverseNahmtransformcovariantderivative}
\begin{split}
&\langle u_x(f), \nabla_x u_x(f')\rangle \\
=& \frac{1}{4} \int_{X^\vee} \sum_j \langle \epsilon ( G\Psi(f)|_\tau (\psi^j_\tau) ),
\nabla_x \epsilon (   G\Psi(f')|_\tau (\psi^j_\tau)      ) \rangle_{S^-_{X^\vee|_\tau \otimes \mathcal{E}|_{x,\tau} } } d\text{Vol}_{X^\vee}  \\
=& \frac{1}{4} \int_{X^\vee} \sum_j \langle
\nabla_x (   G\Psi(f')|_\tau (\psi^j_\tau)      ),   G\Psi(f)|_\tau (\psi^j_\tau)  \rangle_{S^-_{X^\vee|_\tau \otimes \mathcal{E}|_{x,\tau} } } d\text{Vol}_{X^\vee} .
\end{split}
\end{equation}
In our situation $\langle f, \nabla^\alpha f'\rangle =0$. By a small variant of the argument of Lemma \ref{HermitianisometryLemma1}, the expression (\ref{inverseNahmtransformcovariantderivative}) is equal to
\[
\frac{1}{4}\Tr_{ \mathcal{E^\vee}|_{x,\tau}     } \int_{X^\vee} \langle f'\circ \nabla_x \circ \Tr_{S^-_{X^\vee}|_\tau} G_\tau \Omega^t\cdot P_\tau (\Omega^t\cdot)^\dag G_\tau , f \rangle d\text{Vol}_{X^\vee} .
\]
By Lemma \ref{HermitianisometryLemma2} and the argument of Lemma \ref{HermitianisometryLemma3}, this is
\begin{equation}\label{inverseNahmtransformcovariantderivative2}
\lim_{y\to x} \int_{X^\vee} \Tr_{ \mathcal{E^\vee}|_{x,\tau} } \langle f'|
Q_\tau (x,y) \nabla^{\alpha_\tau}_y \Lap_{X^\vee}  G_\tau(y,x) |f\rangle  d\text{Vol}_{X^\vee}.
\end{equation}
Here the notation $\nabla^{\alpha_\tau}_y$ means $\sum_{\mu } dy_\mu \nabla^{\alpha_\tau}_{  \frac{\partial}{\partial y_\mu }      }$.

By computing commutators, 
\[
\begin{split}
\nabla^{\alpha_\tau}_{ \frac{\partial}{\partial y_i }  } \Lap_{X^\vee}=& 
\Lap_{X^\vee}\nabla^{\alpha_\tau }_{ \frac{\partial}{\partial y_i } }
-\sum_j \nabla^{univ,t }_{\frac{\partial}{\partial \tau_j}   }(  \Omega^t(   \frac{\partial}{\partial y_i }, \frac{\partial}{\partial \tau_j}  )         )-2\sum_j \Omega^t(  \frac{\partial}{\partial y_i }, \frac{\partial}{\partial \tau_j}      )\nabla^{univ ,t }_{  \frac{\partial}{\partial \tau_j}      } \\
= &\Lap_{X^\vee}\nabla^{\alpha_\tau }_{ \frac{\partial}{\partial y_i }}
+2\sum_j \Omega^t(  \frac{\partial}{\partial y_i }, \frac{\partial}{\partial \tau_j}      )(-\nabla^{univ, t}_{  \frac{\partial}{\partial \tau_j}      })
\end{split}
\]
where the second equality uses the Coulumb condition (\ref{YangMillsoncomponents}), and the minus transpose is inserted to remember the fact that $G_\tau(y,x)$ involves the dualised factor $\mathcal{E}^\vee|_{y,\tau}$. Using this and the Green's formula, the expression (\ref{inverseNahmtransformcovariantderivative2}) is
\[
\begin{split}
\lim_{y\to x} \int_{X^\vee} \Tr_{ \mathcal{E^\vee}|_{x,\tau} } & \{ \langle f'| (\Lap_{X^\vee}
Q_\tau (x,y) ) \nabla^{\alpha_\tau }_y   G_\tau(y,x) |f\rangle \\
& -2\sum_{i,j} dy_i \langle f'| (\nabla^{univ}_{\frac{\partial}{\partial \tau_j} } Q_\tau(x,y) )
\circ \Omega^t(  \frac{\partial}{\partial y_i }, \frac{\partial}{\partial \tau_j} )G_\tau(y,x) |f \rangle 
\}
 d\text{Vol}_{X^\vee}
.
\end{split}
\]
The sign in front of the second terms comes from integration by part.

By the asymptotic formula of $G_\tau(y,x)$ (\cf (\ref{asymptotesGreenfunction})), we can replace $G_\tau(y,x)$ by $\frac{I}{4\pi^2|x-y|^2 }$, and $\nabla^{\alpha_\tau }_y   G_\tau(y,x)$
by $- \frac{2\sum_i y_i dy_i I}{ 4\pi^2|x-y|^4   }$, because the lower order terms in the asymptotic expansion are smooth enough to be neglegible when $y\to x$. Thus in the trivialisation given by $Q_\tau$ as in Section \ref{Injectiveisometry}, the quantity
$\langle u_x (f), \nabla_{\frac{\partial}{\partial x_\mu}   }u_x(f')\rangle$ is equal to
\begin{equation}\label{inverseNahmtransformcovariantderivative3}
\begin{split}
\lim_{y\to x} \frac{-1}{2\pi^2|x-y|^2} & \int_{X^\vee}    \{  \frac{y_\mu}{|x-y|^2 }\langle f'|\Tr_{ \mathcal{E}|_{x,\tau} } (\Lap_{X^\vee}
Q_\tau (x,y) )  |f\rangle \\
& +\sum_{j} \langle f'|\Tr_{ \mathcal{E}|_{x,\tau} } ( \nabla^{univ}_{\frac{\partial}{\partial \tau_j} } Q_\tau(x,y)\Omega(  \frac{\partial}{\partial y_\mu }, \frac{\partial}{\partial \tau_j} ))   |f \rangle 
\}
d\text{Vol}_{X^\vee}
.
\end{split}
\end{equation}
Here we $\Omega^t$ has been changed to $\Omega$ when we convert trace over $\mathcal{E}^\vee|_{x,\tau}$ to trace over $\mathcal{E}|_{x,\tau}$. The trivialisation $Q_\tau$ is used to identify $\mathcal{E}|_{y,\tau}$ with $\mathcal{E}|_{x,\tau}$.

We then proceed to the evaluation of this limit (\ref{inverseNahmtransformcovariantderivative3}). This requires us to know  asymptotic expansions to one higher order compared to Section \ref{Injectiveisometry}.

\begin{lem}(Higher order asymptotes)\label{HigherorderasymptotesLemma}
In the geodesic coordinates on $X$, under the trivialisation induced by $Q_\tau$, we have the asymptotic formula as $y\to x$,
\begin{equation}\label{thirdorderTaylorasymptote0}
\frac{1}{4\pi^2|x-y|^2}\int_{X^\vee}
\Tr_{ \mathcal{E}|_{x,\tau} }\langle f'|  \Lap^{univ}_{X^\vee} Q_\tau(x,y) |f\rangle d\text{Vol}_{X^\vee}\sim \langle f', f\rangle +O(|y-x|^2).
\end{equation}
Morever,
\begin{equation}\label{thirdorderTaylorasymptote0'}
\begin{split}
&\frac{1}{4\pi^2}\int_{X^\vee} \langle f'|\Tr_{ \mathcal{E}|_{x,\tau} } \{ \sum_i (\nabla^{univ}_{\frac{\partial}{\partial \tau_i}  } Q_\tau(x,y)  ) \Omega( \frac{\partial }{\partial y_\mu} , \frac{\partial }{\partial \tau_i} )    \}
|f\rangle d\text{Vol}_{X^\vee}\\
&\sim -y_\mu \langle f', f\rangle   + O(|y-x|^3).
\end{split}
\end{equation}	
\end{lem}

Given (\ref{thirdorderTaylorasymptote0}), (\ref{thirdorderTaylorasymptote0'}), we immediately see that the limit (\ref{inverseNahmtransformcovariantderivative3}) vanishes by a cancellation of asymptotic expressions. This shows
\[
\langle u_x (f), \nabla_{\frac{\partial}{\partial x_\mu}   }u_x(f')\rangle=0
\]
as required.
\end{proof}

Now we show the higher order asymptotes.

\begin{proof}(of Lemma \ref{HigherorderasymptotesLemma})
To save some writing, we will use the summation convention, and furthermore we will simply write $\nabla_k$ for $\nabla^{Q_\tau}_{ \frac{\partial}{\partial y_k }    }	$. The trace will be over the $\mathcal{E}$ factor using the trivialisation $Q_\tau$. Starting from (\ref{QtauLaplacianasymptote1}),
\[
\begin{split}
& \nabla_k \nabla_j \Lap^{univ}_{X^\vee} Q_\tau(x,y) 
= (\nabla_k\Lap^{univ}_{X^\vee}  A_j )Q_\tau(x,y) \\
&- 2  \{
\Omega(\frac{\partial}{\partial y_j},  { \frac{\partial}{\partial \tau_i}  }  )
+ \nabla^{univ}_{ \frac{\partial}{\partial \tau_i}  }A_j \} \{
\Omega(\frac{\partial}{\partial y_k},  { \frac{\partial}{\partial \tau_i}  }  )
+ \nabla^{univ}_{ \frac{\partial}{\partial \tau_i}  }A_k\} Q_\tau(x,y) \\
& -2\nabla_k \{  \Omega(\frac{\partial}{\partial y_j},  { \frac{\partial}{\partial \tau_i}  }  )
+ \nabla^{univ}_{ \frac{\partial}{\partial \tau_i}  }A_j        \}
\nabla^{univ}_{ \frac{\partial}{\partial \tau_i}  } Q_\tau(x,y) ,
\end{split}
\]
so taking one more derivative and evaluating at $y=x$, the third order derivative is given by
\begin{equation}\label{thirdorderTaylorasymptote1}
\begin{split}
& \nabla_l \nabla_k \nabla_j \Lap^{univ}_{X^\vee} Q_\tau(x,y) |_{y=x}
= \Lap^{univ}_{X^\vee} \nabla_l\nabla_k A_j  - 2\nabla_l \{
\Omega(\frac{\partial}{\partial y_j},  { \frac{\partial}{\partial \tau_i}  }  )\Omega(\frac{\partial}{\partial y_k},  { \frac{\partial}{\partial \tau_i}  }  )\}
\\
& -2\{  ( \nabla^{univ}_{ \frac{\partial}{\partial \tau_i}  } \nabla_l A_j    ) \Omega(\frac{\partial}{\partial y_k},  { \frac{\partial}{\partial \tau_i}  }  )
+ \Omega(\frac{\partial}{\partial y_j},  { \frac{\partial}{\partial \tau_i}  }  ) \nabla^{univ}_{ \frac{\partial}{\partial \tau_i}  } \nabla_l A_k \}
\\
& -2\{ \nabla_k\Omega(\frac{\partial}{\partial y_j},  { \frac{\partial}{\partial \tau_i}  }  )
+ \nabla^{univ}_{ \frac{\partial}{\partial \tau_i}  } \nabla_k A_j           \} \Omega( \frac{\partial}{\partial y_l},  { \frac{\partial}{\partial \tau_i}  }  ).
\end{split}
\end{equation}

We then integrate $\Tr (\nabla_l \nabla_k \nabla_j \Lap^{univ}_{X^\vee} Q_\tau(x,y) |_{y=x})$ with respect to $\tau$ over $X^\vee$. Many terms do not contribute to the integral: for example,
\[
\int_{X^\vee} \Tr \Lap^{univ}_{X^\vee} \nabla_l\nabla_k A_j d\text{Vol}_{X^\vee}=
\int_{X^\vee} \Lap_{X^\vee} \Tr \nabla_l\nabla_k A_j d\text{Vol}_{X^\vee}=
0,
\]
because the integral of the Laplacian of a function is zero. For another example,
\[
\begin{split}
& \int_{X^\vee}
\Tr ( \nabla^{univ}_{ \frac{\partial}{\partial \tau_i}  } \nabla_l A_j    ) \Omega(\frac{\partial}{\partial y_k},  { \frac{\partial}{\partial \tau_i}  }  ) d\text{Vol}_{X^\vee} \\
=&  
\int_{X^\vee}
 \Tr \nabla^{univ}_{ \frac{\partial}{\partial \tau_i}  } \{(\nabla_l A_j)     \Omega(\frac{\partial}{\partial y_k},  { \frac{\partial}{\partial \tau_i}  }  )\} d\text{Vol}_{X^\vee} \\
=&  
 \int_{X^\vee}
 \text{div}_{X^\vee}\Tr  \{(\nabla_l A_j)     \Omega(\frac{\partial}{\partial y_k},  { \frac{\partial}{\partial \tau_i}  }  ) d\tau_i\} d\text{Vol}_{X^\vee}
 =0,
 \end{split}
\]
because the integral of the divergence of a 1-form is zero. Summing up all contributions, one obtains
\begin{equation}\label{thirdorderTaylorasymptote2}
\begin{split}
& \int_{X^\vee}
\Tr \nabla_l \nabla_k \nabla_j \Lap^{univ}_{X^\vee} Q_\tau(x,y) |_{y=x}d\text{Vol}_{X^\vee}\\
= &-2 \int_{X^\vee} \Tr \nabla_l \{
\Omega(\frac{\partial}{\partial y_j},  { \frac{\partial}{\partial \tau_i}  }  )\Omega(\frac{\partial}{\partial y_k},  { \frac{\partial}{\partial \tau_i}  }  )\}+ \Tr \nabla_k\Omega(\frac{\partial}{\partial y_j},  { \frac{\partial}{\partial \tau_i}  }  )\Omega( \frac{\partial}{\partial y_l},  { \frac{\partial}{\partial \tau_i}  }  ) d\text{Vol}_{X^\vee} \\
= & -2  \int_{X^\vee} \Tr \nabla_k\Omega(\frac{\partial}{\partial y_j},  { \frac{\partial}{\partial \tau_i}  }  )\Omega( \frac{\partial}{\partial y_l},  { \frac{\partial}{\partial \tau_i}  }  ) d\text{Vol}_{X^\vee} +8\pi^2 \frac{\partial g_{jk} }{\partial y_l},
\end{split}
\end{equation}
where the last equality uses
\[
-\frac{1}{4\pi^2}\int_{X^\vee} \Tr \{
\Omega(\frac{\partial}{\partial y_j},  { \frac{\partial}{\partial \tau_i}  }  )\Omega(\frac{\partial}{\partial y_k},  { \frac{\partial}{\partial \tau_i}  }  )\}d\text{Vol}_{X^\vee}= g(\frac{\partial}{\partial y_j}, \frac{\partial}{\partial y_k}  )=g_{jk},
\]
as in the proof of Lemma \ref{HermitianisometryLemma4}. We notice the LHS of (\ref{thirdorderTaylorasymptote2}) is symmetric in $l,k,j$, because in a trivialisation the partial derivatives commute. If we switch $j,l$, we obtain
\[
\begin{split}
& \int_{X^\vee}
\Tr \nabla_l \nabla_k \nabla_j \Lap^{univ}_{X^\vee} Q_\tau(x,y) |_{y=x}d\text{Vol}_{X^\vee}\\
=& -2  \int_{X^\vee} \Tr \nabla_k\Omega(\frac{\partial}{\partial y_l},  { \frac{\partial}{\partial \tau_i}  }  )\Omega( \frac{\partial}{\partial y_j},  { \frac{\partial}{\partial \tau_i}  }  ) d\text{Vol}_{X^\vee} +8\pi^2 \frac{\partial g_{lk} }{\partial y_j}.
\end{split}
\]
Adding this to (\ref{thirdorderTaylorasymptote2}), and divide by $8\pi^2$, we get
\begin{equation}\label{thirdorderTaylorasymptote3}
\frac{1}{4\pi^2}\int_{X^\vee}
\Tr \nabla_l \nabla_k \nabla_j \Lap^{univ}_{X^\vee} Q_\tau(x,y) |_{y=x}d\text{Vol}_{X^\vee}
= \frac{\partial g_{jk} }{\partial y_l}+ \frac{\partial g_{lk} }{\partial y_j}+ \frac{\partial g_{lj} }{\partial y_k}.
\end{equation}
Substituting this into (\ref{thirdorderTaylorasymptote2}), we get
\begin{equation}\label{thirdorderTaylorasymptote4}
\int_{X^\vee} \Tr \nabla_k\Omega(\frac{\partial}{\partial y_j},  { \frac{\partial}{\partial \tau_i}  }  )\Omega( \frac{\partial}{\partial y_l},  { \frac{\partial}{\partial \tau_i}  }  )|_{y=x} d\text{Vol}_{X^\vee}
=-4\pi^2 g_{ls} \Gamma^s_{jk}, 
\end{equation}
where $\Gamma^s_{jk}$ is the Christoffel symbol on $X$ at $x\in X$.

Using (\ref{thirdorderTaylorasymptote3}), we can improve Lemma \ref{HermitianisometryLemma4} to the next order. In our coordinates, the point $x$ corresponds to the origin, and $|y-x|^2=g_{ij}(x)y_iy_j$. Then
\begin{equation}\label{thirdorderTaylorasymptote5}
\begin{split}
&\frac{1}{4\pi^2|x-y|^2}\int_{X^\vee}
\Tr  \Lap^{univ}_{X^\vee} Q_\tau(x,y) d\text{Vol}_{X^\vee} \\
& \sim I_{\mathcal{F}|_x}(1+ \frac{1}{2|x-y|^2}y_jy_ky_l \frac{\partial g_{lk} }{\partial y_j} ) +O(|y-x|^2).
\end{split}
\end{equation}
In the geodesic coordinate, the third order derivative term vanishes, and we obtain (\ref{thirdorderTaylorasymptote0}), as required.

Now we show the second part of the Lemma.
We can do a similar but easier computation as before by calculating $\nabla^{univ}_{\frac{\partial}{\partial \tau_i}  } Q_\tau$ up to the second order. After discarding some Laplacian and divergence terms in the $X^\vee$ integration, we get
\[
\begin{split}
&\int_{X^\vee} \Tr \{ (\nabla_k\nabla^{univ}_{\frac{\partial}{\partial \tau_i}  } Q_\tau(x,y)  ) \Omega( \frac{\partial }{\partial y_\mu} , \frac{\partial }{\partial \tau_i} )    \}|_{y=x} d\text{Vol}_{X^\vee} \\
=&\int_{X^\vee} \Tr \{ \Omega( \frac{\partial }{\partial x_k} , \frac{\partial }{\partial \tau_i} )  \Omega( \frac{\partial }{\partial x_\mu} , \frac{\partial }{\partial \tau_i} )    \}d\text{Vol}_{X^\vee}= -4\pi^2 g_{k\mu},
\end{split}
\]
and
\[
\begin{split}
&\int_{X^\vee} \Tr \{ (\nabla_j\nabla_k\nabla^{univ}_{\frac{\partial}{\partial \tau_i}  } Q_\tau(x,y)  ) \Omega( \frac{\partial }{\partial y_\mu} , \frac{\partial }{\partial \tau_i} )    \}|_{y=x} d\text{Vol}_{X^\vee} \\
=&\int_{X^\vee} \Tr \{ \nabla_j\Omega( \frac{\partial }{\partial y_k} , \frac{\partial }{\partial \tau_i} )  \Omega( \frac{\partial }{\partial y_\mu} , \frac{\partial }{\partial \tau_i} )    \}|_{y=x} d\text{Vol}_{X^\vee}= -4\pi^2 g_{s\mu}\Gamma^s_{jk},
\end{split}
\]
where the last equality uses (\ref{thirdorderTaylorasymptote4}). From this, the second derivative 
\[
\begin{split}
&\int_{X^\vee} \frac{\partial^2}{\partial y_j \partial y_k} \Tr \{ (\nabla^{univ}_{\frac{\partial}{\partial \tau_i}  } Q_\tau(x,y)  ) \Omega( \frac{\partial }{\partial y_\mu} , \frac{\partial }{\partial \tau_i} )    \}|_{y=x} d\text{Vol}_{X^\vee} \\
=& -4\pi^2 g_{s\mu}\Gamma^s_{jk}
+\int_{X^\vee} \Tr \{ \Omega( \frac{\partial }{\partial y_k} , \frac{\partial }{\partial \tau_i} ) \nabla_j \Omega( \frac{\partial }{\partial y_\mu} , \frac{\partial }{\partial \tau_i} )    \}|_{y=x} d\text{Vol}_{X^\vee} \\
& + \int_{X^\vee} \Tr \{ \Omega( \frac{\partial }{\partial y_j} , \frac{\partial }{\partial \tau_i} ) \nabla_k \Omega( \frac{\partial }{\partial y_\mu} , \frac{\partial }{\partial \tau_i} )    \}|_{y=x} d\text{Vol}_{X^\vee}  \\
=& -4\pi^2 g_{s\mu}\Gamma^s_{jk}-4\pi^2 g_{sk}\Gamma^s_{j\mu}-4\pi^2 g_{sj}\Gamma^s_{\mu k} \\
=& -2\pi^2\{ \frac{\partial g_{jk}}{\partial y_\mu} + \frac{\partial g_{\mu k}}{\partial y_j} + \frac{\partial  g_{j\mu }}{\partial y_k}          \}.
\end{split}
\]
Combining these, we get
\begin{equation}
\begin{split}
&\frac{1}{4\pi^2}\int_{X^\vee} \Tr \{ (\nabla^{univ}_{\frac{\partial}{\partial \tau_i}  } Q_\tau(x,y)  ) \Omega( \frac{\partial }{\partial y_\mu} , \frac{\partial }{\partial \tau_i} )    \} d\text{Vol}_{X^\vee}\\
&\sim I_{\mathcal{F}|_x}\{-g_{k\mu}(x) y_k- \frac{1}{4} (  \frac{\partial g_{jk}}{\partial y_\mu} + \frac{\partial g_{\mu k}}{\partial y_j} + \frac{\partial  g_{j\mu }}{\partial y_k}     )(x) y_jy_k \}+ O(|y-x|^3).
\end{split}
\end{equation}
In geodesic coordinates, we have that $g_{s\mu}(x)=\delta_{s\mu}$ and the Christoffel symbols vanish, so (\ref{thirdorderTaylorasymptote0'}) follows.
\end{proof}

\begin{rmk}
We have worked  in the geodesic coordinate on $X$, which has the advantage of simplifying the asymptotic formulae for $\Lap_{X^\vee} Q_\tau$ and $G_\tau$. It is an interesting exercise to show that even if we use more general coordinates and keep the Christoffel symbols, the asymptotic formulae still cancel out exactly in the proof of Proposition \ref{comparingconnectionmatrices}.
\end{rmk}

We now collect the main results of Section \ref{inverseNahmtransformcomparisonmap}, \ref{Injectiveisometry}, \ref{Comparingtheconnections} to achieve

\begin{thm}\label{Fourierinversiontheorem}
(Fourier inversion) Assume the setup of Section \ref{inverseNahmtransformcomparisonmap} so that the inverse Nahm transform $(\hat{\hat{\mathcal{F}}},\hat{\hat{\alpha}})$ is well defined. Then the canonical comparison map $\mathcal{F}\to \hat{\hat{\mathcal{F}}}$ is an isometric isomorphism of Hermitian vector bundles, which identifies the connection $\alpha$ with $\hat{\hat{\alpha}}$.
\end{thm}

\begin{rmk}
The reader may compare this to Proposition 10 and Corollary 3 in \cite{Bartocci2}, which shows a derived category version of Fourier inversion theorem. Their version is technically much simpler, due to the power of the derived category machinery.
\end{rmk}

\begin{rmk}
Our proof depends on the algebraic theory in \cite{Huy1} only by the use of the Fourier-Mukai transform on cohomology.
\end{rmk}

\section*{Appendix: spinors in dimension 4}

We review the well known linear algebraic model of spinors in dimension 4, which serves to establish the conventions used in this paper.

Let $S$ be the spin representation of $so(4)$, which admits the chiral splitting into positive and negative spinors
$
S=S^+ \oplus S^-.
$ For a concrete model of S, we take a complex two dimensional Hermitian vector space $W$, with orthonormal basis $f_1,f_2 $, and let $S=\Lambda^*(W)$. Then $S^+=\C \oplus \Lambda^2(W)$ and $S^-=W$. The Clifford multiplication of an orthonormal basis $\frac{\partial}{\partial x_i}$ of $\R^4$ on $S$ is given by 
\[
\begin{cases}
c_1=f_1\wedge -f_1 \angle \\
c_2=if_1\wedge +if_1 \angle \\
c_3=f_2\wedge -f_2 \angle \\
c_4=if_2\wedge +if_2 \angle 
\end{cases}
\]
We see $c_i \in \Hom(S_+, S_-)\oplus \Hom(S_-, S_+)$, $c_i^\dagger=-c_i$, and the Clifford relations $c_ic_j+c_jc_i=-2\delta_{ij}$. The Clifford multiplication is multiplicative on norms:
\[
|c(v)\cdot \xi |=|v| |\xi|, \quad |c(v)\cdot \eta |=|v| |\eta|, \quad v\in \R^4, \xi\in S^+, \eta\in S^-.
\]
%The chiral splitting can also be described by introducing the operator $\gamma_0=c_1c_2c_3c_4$, which is $-1$ on $S^+$ and 1 on $S^-$. 
With this choice of convention, the Dirac operator $D=c_i \nabla_{\frac{\partial}{\partial x_i}}$ is self adjoint. It is conventional to think of $D$ as a pair of formally adjoint operators, mapping between the positive and negative spinor bundles. 

It is also convenient to extend the Clifford multiplication to elements of $\Lambda^2(\R^4)$. For example, for $F=\sum_{i<j}{ F_{ij} dx^i \wedge dx^j   }$, we make it act as $\sum_{i<j}{ F_{ij} c_ic_j }$. This Clifford multiplication and the wedge product can be compared by the formula
\[
c(v)c(w)=c( v\wedge w)- (v, w), \quad v, w\in \R^{4*}.
\]
Self dual forms act only on positive spin, and ASD forms only act on negative spin. This  applies to the standard triple of hyperk\"ahler 2-forms 
\[
\omega_1 =dx_1 dx_2+ dx_3 dx_4, \quad
\omega_2 =dx_1 dx_3 + dx_4 dx_2, \quad
\omega_3 =dx_1 dx_4+ dx_2 dx_3.
\]

 The group $Spin(4)$ acts on $\R^4, S^+, S^-$, such that the Clifford multiplication map $\R^4\times S^{+}\to S^{-}$ is equivariant. Up to $2:1$ cover, we can think of standard complex structures $I_1, I_2, I_3 \in SO(4)$ as elements of $Spin(4)$, and as such they act on the spinors. The action of $I_k$ on $S^-$ can be chosen to be trivial; this uniquely specifies the action $I_k^{S^+}$ on $S^+$, which must satisfy the compatibility with Clifford multiplication:
 \begin{equation}\label{complexstructureactiononpositivespin1}
 c(I_k v)\cdot{} I_k^{S^+} \xi= c(v) \cdot{}\xi, \quad \xi\in S^+, v\in \R^4,
 \end{equation}
 and
 \begin{equation}\label{complexstructureactiononpositivespin2}
 I_k^{S^+} (c(v)\cdot{} \eta)= c(I_k v)\cdot{} \eta, \quad \eta\in S^-, v\in \R^4.
 \end{equation}
 One can more concretely think of $I_k^{S^+}$ as given by the matrices in the basis $\{1, f_1\wedge f_2\}$,
 \[
I_1^{S^+} = \begin{bmatrix}
-i       & 0 \\
0       & i \\
\end{bmatrix}, 
I_2^{S^+} = \begin{bmatrix}
0      & -1 \\
1       & 0 \\
\end{bmatrix},
I_3^{S^+} = \begin{bmatrix}
0      & i \\
i       & 0 \\
\end{bmatrix}.
 \]
In particular they satisfy the quaternionic algebraic relations. On a hyperk\"ahler 4-fold, these actions globalise to actions on the positive spin bundle. The action on the negative spin bundle is trivial. Another way to understand these actions $I_i^{S^+}$, is that it equals a half of Clifford multiplication by the 2-form $\omega_i$.

There is an antilinear symmetry of $S$ as a Clifford module, coming from the $SU(2)$ structure on $S^+$ and $S^-$, given explicitly in our model by $\epsilon: S\rightarrow S$,
\[
%\begin{cases}
f_1 \mapsto f_2, f_2 \mapsto -f_1,
1 \mapsto -f_1\wedge f_2, f_1\wedge f_2 \mapsto 1.
%\end{cases}
\]
It satisfies $\epsilon^2=-1$, and commutes with Clifford multiplication and the $I^{S^+}_k$ action. On a hyperk\"ahler manifold, this antilinear symmetry can be globalised to a covariantly constant structure. Another viewpoint on $\epsilon$ is that together with the Hermitian structure it induces the complex symplectic form $\langle \langle \_, \_\rangle \rangle=\langle  \epsilon\_, \_ \rangle$ on $S^+$ or $S^-$.

\end{document}